\documentclass{interact}
\usepackage{epstopdf}% To incorporate .eps illustrations using PDFLaTeX, etc.
\usepackage[caption=false]{subfig}% Support for small, `sub' figures and tables
\usepackage[numbers,sort&compress]{natbib}% Citation support using natbib.sty
\bibpunct[, ]{[}{]}{,}{n}{,}{,}% Citation support using natbib.sty
\makeatletter% @ becomes a letter
\def\NAT@def@citea{\def\@citea{\NAT@separator}}% Suppress spaces between citations using natbib.sty
\makeatother% @ becomes a symbol again

\theoremstyle{plain}% Theorem-like structures provided by amsthm.sty
\newtheorem{theorem}{Theorem}[section]
\newtheorem{lemma}[theorem]{Lemma}
\newtheorem{corollary}[theorem]{Corollary}
\newtheorem{proposition}[theorem]{Proposition}

\theoremstyle{definition}
\newtheorem{definition}[theorem]{Definition}
\newtheorem{example}[theorem]{Example}

\theoremstyle{remark}
\newtheorem{remark}{Remark}

\usepackage{amsmath}
\usepackage{amssymb}
\usepackage[bookmarksnumbered]{hyperref}
\usepackage{comment}

\usepackage{amsthm}
\usepackage{graphicx}
\usepackage{float}
\usepackage{color}
\usepackage{enumitem}
\renewcommand{\labelenumi}{(\roman{enumi})}

\usepackage{soul}
\usepackage{pgfplots}
\usepackage{marginnote}
\usepackage[utf8]{inputenc}
\usepackage[hang,flushmargin]{footmisc} 
\reversemarginpar

\begin{document}
\articletype{ARTICLE TEMPLATE}% Specify the article type or omit as appropriate

\title{Duality for Composite Optimization Problem within the Framework of Abstract Convexity}

\author{
\name{Ewa Bednarczuk\textsuperscript{a}\thanks{CONTACT The Hung Tran. Email:\ tthung@ibspan.waw.pl \\
\textbf{Disclaimer}: This work represents only the author’s view and the European Commission is not responsible for any use that may be made of the information it contains.} and Hung T.T.\textsuperscript{a}}
\affil{\textsuperscript{a}Department of Modeling and Optimization in Dynamical System, System Research Institute-Polish Academy of Science, Warsaw, Poland}
}

\maketitle

\begin{abstract}
We study conjugate and Lagrange dualities for  composite optimization problems within the framework of abstract convexity.
We provide conditions for zero duality gap in conjugate duality. For Lagrange duality, intersection property is applied to obtain zero duality gap. Connection between Lagrange dual and conjugate dual is also established. Examples related to convex and weakly convex functions are given.
\end{abstract}

\begin{keywords}
Abstract convexity; $\Phi$-convexity; $\varepsilon$-subdifferentials; Conjugate dual; Lagrange dual; Zero duality gap; Nonconvex programming; Global Optimization
\end{keywords}
\begin{amscode}
49J27; 49J35; 49N15; 90C46; 90C26
\end{amscode}

\section{Introduction}
During the last eighty years, convexity has become an essential part in the development of optimization, nonlinear analysis etc. With the surge of computational power, interests in machine learning and data science have risen, which make convexity the backbone of many theories and applications. Together with convexity, there are also many works that try to go beyond convexity which include functions
%. These result in convex variants 
like: strongly convex \cite{vial1983strong}, quasi-convex, pseudo-convex, strongly para-convex and para-convex \cite{rolewicz1979paraconvex}, delta-convex \cite{ulam1952approximate}, approximately convex \cite{pales2003approximately} etc...
%Among them, abstract convexity is considered to be more general than others, ....

Abstract convexity encompasses many of the above mentioned classes of functions. Abstract convexity, presented in the monographs of Rubinov \cite{Rub2013}, Pallaschke and Rolewicz \cite{Pall2013} is based on the idea to bring convexity outside the range of linearity, by introducing the nonlinear environment where a function becomes an upper envelope of a subset of functions with specific rules. The term “abstract convexity” was used by Rubinov \cite{Rub2013} to describe functions which are upper envelopes of a given class of function $\Phi$, i.e.
\[
f(x) = \sup_{\phi\in \Phi} \left\{\phi(x)\ :X\to\mathbb{R}, \ \phi \leq f\right\},
\]
where $X$ is a nonempty set.

Consequently, the concept of abstract convexity works well with classical convex functions. As the supremum operation is retained, many global properties of convex analysis are still in effect for abstract convexity. Till now, the books of Pallaschke and Rolewicz \cite{Pall2013}, and the monograph of Singer \cite{Sin1997} gathered many results about abstract convexity, %and 
%its framework 
especially about conjugation, subdifferential and duality. They also have historical results of the main ideas of abstract convexity and its applications. 
%While 
In \cite{Rub2013}, Rubinov presented basic notions of abstract convexity and applications to global optimization problems.

It has been fifty years since the theory of abstract convexity started. It has made great progress in many disciplines with applications in both theory and computation. In \cite{Rubinov1999}, the authors investigated the class of increasing star-shaped functions and constructed an algorithm to solve the global optimization problem using abstract convexity. They also generalized the cutting plane algorithm for nonconvex global minimization problems \cite{Rubinov1999,andra2002}. While in \cite{burachik2008,dutta2004,eber2010}, a general version of monotone operator is developed based on abstract convexity.

In this paper, we pay special attention to duality theory for composite minimization problem by using abstract convexity. 
We mainly study the conjugate and Lagrangian dualities within the  abstract convexity for the composite minimization problem,  
\begin{equation}
\label{prob:CP} 
\inf_{x\in X} f(x) +g(Lx), \tag{CP}
\end{equation}
where $X$ and $Y$ are the domain sets (spaces) of the functions $f$ and $g$, the mapping $L:X\to Y$ maps $X$ into $Y$.

The composite optimization problem is taken as our main interest as it is a general problem of many optimization formulation. One can think of \eqref{prob:CP} as nonlinear programming problem, min-max optimization, constrained and unconstrained problem. These problems have been studied extensively in \cite{rockafellar2009variational,Bon2013}. On the other hand, composite problem can also be considered as a regularized problem by treating the $g(Lx)$ as a penalty term. Such problems are knows as total variation model in image deblurring and denoising \cite{beck2009fast}.

Several efforts have been made for the last twenty years to investigate dualities for optimization problem within the framework of abstract convexity.
In \cite{Jey2007} 
strong duality is proved for infimal convolution of Fenchel’s duality. In \cite{Burachik2007}, zero duality gap and exact multiplier of augmented Lagrangian are investigated by using the framework of abstract convexity. While in \cite{Syga2016,Syga2018} necessary and sufficient  conditions are given to achieve minimax equality from a general Lagrangian using the definition of abstract convexity. In the works of Dolgopolik \cite{Dolgo2015}, Gorokhovik and Tykoun \cite{goro2020,goro20201}, the authors applied abstract convexity to approximate the class of nonsmooth functions and formulate necessary optimal conditions for global nonsmooth nonconvex problem. The authors in \cite{Bui2021} investigated the problem of minimizing the finite sum of arbitrary functions and provided conditions for zero duality gap through infimal convolution dual. 
%All in all, their results only focus on the set of abstract linear or abstract affine elementary functions with the possibility to generalize to a real-valued function. 
On the other hand, in \cite{Bed2020}, the problem of zero duality gap is studied with the help of perturbation function.
%studied the problem from a different view. 
%instead of considering only one set of elementary functions, it is assumed that each function can have a different set of elementary functions and studied the 
%relationship between these sets. 

%Following the work of \cite{Bed2020}, 
%. Instead of considering the same variable space, we examine the case where there are different variable spaces in 

\textbf{Our contribution is as follow}.
Inspired by Moreau's general subdifferential and conjugation \cite{moreau1970}, we utilize the framework of abstract conjugation and abstract subdifferential \cite{Jey2007} and construct conjugate dual problem to \eqref{prob:CP}. We derive conditions  for zero duality gap and strong duality. As we have different variable spaces of $f$ and $g$ , the conditions we propose reduce to the ones proposed 
in e.g. \cite{Bui2021} when $L$ is the identity mapping. 
In deriving the Lagrangian dual problem 
we make use of the intersection property for abstract convex functions, \cite{Bed2020,Syga2018} to obtain Lagrange zero duality gap. We also discuss the relationship between conjugate and Lagrangian duals.

The structure of paper is as follows. In Section \ref{sec:preliminaries}, we state the framework of abstract convexity with conjugation and subdifferential as well as some basic properties for abstract convex functions. Section \ref{sec:Conjugate dual} contains the definition of the composite minimization problem and way to construct the conjugate dual problem. 
Section \ref{sec:zero dual gap} contains our main result about zero duality gap given in  Theorem \ref{thm:CP zero gap v2}. In Section \ref{sec:strong duality}, we provide conditions for strong duality in Theorem \ref{thm:epi-zero gap} and Corollary \ref{cor:epi zero gap}. 
In Section \ref{sec:Lagrange dual}, we prove Lagrange zero duality gap in Theorem \ref{thm:Lagrange intersection}. The equivalence of conjugate dual and Lagrange dual is examined under some conditions in Corollary \ref{cor:Lagrange zero gap best}.

\section{Preliminaries}
\label{sec:preliminaries}
%Let $X$ be Hilbert space, with $\left\langle \cdot,\cdot\right\rangle $
%as the inner product and $\left\langle x,x\right\rangle =\left\Vert x\right\Vert ^{2}$.
Let $X$ be a nonempty set.
A function $f:X\to \left( -\infty,+\infty \right]$ is proper if its domain, denoted by $\text{dom }f=\left\{ x\in X:f\left(x\right)<+\infty\right\}$ is nonempty.

%We define
Let $\Phi=\left\{ \phi:X\to\mathbb{R}\right\} $ be a collection of real-valued functions, which is closed under addition of a constant. The support set of $f$ with respect to $\Phi$ is defined as
\[
\text{supp}_{\Phi}f=\left\{ \phi\in\Phi:f\left(x\right)\geq\phi\left(x\right) \left(\forall x\in X\right)\right\} .
\]

%We also denote 
For $f,g:X\to\left( -\infty,+\infty \right]$, we write $f\leq g\Leftrightarrow f\left(x\right)\leq g\left(x\right)$
for all $x\in X$. Elements of $\Phi$ are called  elementary functions.

\begin{definition} \cite{Rub2013,Pall2013}
\label{def:phi-convex}
A function $f:X\to\left( -\infty,+\infty \right]$ is $\Phi$-convex if
\begin{equation}
f\left(x\right)=\sup_{\phi\in\text{supp}_{\Phi}f}\phi\left(x\right),\ \forall x\in X.
\end{equation}
A function $f$ is $\Phi$-convex at $x_0\in X$ if $f(x_0) = \sup_{\phi\in\text{supp}_{\Phi}f}\phi\left(x_0\right)$.
\end{definition}
Note that when a function $f:X\to (-\infty,+\infty]$ is $\Phi$-convex, then $\text{supp } f$ is nonempty, since otherwise, $f\equiv -\infty$.
%$\Phi$-convexity theory allows to define
 %the notions of $\Phi$-conjugate and %$\Phi$-subdifferential as follow.

\begin{definition} \cite{Rub2013,Pall2013}
\label{def:phi-conjugate}
A $\Phi$-conjugate $f^*:\Phi\to (-\infty,+\infty]$ of $f$ is defined as 
\begin{equation}
f_{\Phi}^{*}\left(\phi\right):=\sup_{x\in X}\left\{ \phi\left(x\right)-f\left(x\right)\right\} .
\end{equation}
Analogously, we can define $\Phi$-biconjugate $f^{**}: X\to (-\infty,+\infty]$ of function $f$ as 
\begin{equation}
f^{**}_\Phi (x):= \sup_{\phi\in\Phi}\left\{ \phi (x) - f^*_{\Phi} (\phi)\right\}.
\end{equation}
%In the case when $\varphi:X\to \mathbb{R} \notin \Phi$, we also define 
%\begin{equation}
%\label{eq:notPhi conjugate}
%^* f (\varphi) :=\sup_{x\in X} \varphi(x) - %f(x)
%\end{equation}
\end{definition}

As in  convex analysis, the following result holds.
\begin{theorem}\cite{Rub2013,Pall2013}
\label{thm:Moreau}
A function $f:X\to\left( -\infty,+\infty \right]$ is $\Phi$-convex if and only if
\begin{equation}
f(x) = f^{**}_\Phi (x),\ \forall x\in X
\end{equation}
\end{theorem}
In view of this theorem, we say that $f$ is $\Phi$-convex at a point $x_0\in X$ if $f^{**}_\Phi (x_0) = f(x_0)$.
\begin{example}
\label{example1}
Let $X$ be a Banach space with the topological dual $X^{*}$ and 
\begin{equation}
\label{eq:Phi conv set}
\Phi_{conv}:=\left\{ \phi:X\to\mathbb{R}\, \mid \, \phi\left(x\right)=\left\langle u,x\right\rangle+c ,u\in X^*,c\in \mathbb{R}\right\}. 
\end{equation}
If a function $f:X\to ( -\infty,+\infty]$ is $\Phi_{conv}$-convex, then it is proper, convex and lower semi-continuous. When $X$ is a locally convex space, \cite{ekeland1999}, then $f$ is $\Phi_{conv}$-convex if and only if $f$ is proper, convex and lower semi-continuous.
\end{example}

\begin{example}
\label{example3}
Let $X$ be a Hilbert space with the inner product $\left\langle \cdot,\cdot\right\rangle$ and the norm $\left\langle x,x\right\rangle =\left\Vert x\right\Vert ^{2}$. Let $a\in\mathbb{R}$, take $\Phi_{Q,a}$ as a collection of quadratic functions 
\begin{equation}
\label{eq:Phi a set}
\Phi_{Q,a}:=\left\{ \phi:X\to\mathbb{R} \ |\  \phi\left(x\right)=a\left\Vert x\right\Vert ^{2}+\left\langle u,x\right\rangle +c \text{ where } c\in\mathbb{R},u\in X\right\}.
\end{equation}
A function $f:X\to (-\infty,+\infty]$ is $\Phi_{Q,a}$-convex, if for every $x\in X$, we have 
\[
f(x) = \sup_{\phi \in \Phi_{Q,a}} \{ \phi (x) \ | \ \phi  \leq f \}.
\]
Since $\Phi_{Q,a}$ consists of continuous functions, $f$ is lower semi-continuous on $X$.
%with $a\leq 0$ coincides with the class of lsc functions defined on $X$, see \cite[Proposition 6.3]{Rub2013}. 
Depending upon the sign of $a$, we get the following.
\begin{itemize}
\item If $a=0$ then we go back to the affine elementary functions as in Example \ref{example1}, so $f$ is lsc proper convex if $f$ is $\Phi_{Q,a}$-convex.
\item If $a>0$, $f$ is $\Phi_{Q,a}$-convex, then $f$ is strongly convex with modulus $a$. One can look at the definition of a strongly convex function in \cite[Definition 4.1]{vial1983strong}.
\item If $a<0$ and  $f$ is $\Phi_{Q,a}$-convex, then 
%for fixed $a$, we call 
$f$ is a weakly convex function with modulus $a$  \cite[Definition 4.1]{vial1983strong}, see also \cite{rolewicz1979paraconvex}.
\end{itemize}

For equivalent definitions of weakly convex functions and related facts, see e.g. \cite{cannarsa2004,attouch1993approximation} and the references therein. 
\end{example}

\begin{definition}
\label{def:weak convex}
Let $X$ be a normed space with the norm $\lVert \cdot\rVert$. A function $f:X\to\left( -\infty,+\infty \right]$ is weakly convex with modulus $\rho\geq0$ or $\rho$-weakly convex if $f+\rho\left\Vert \cdot\right\Vert ^{2}$ is a convex function. When $\rho =0$, $f$ becomes a convex function.
\end{definition}
Note that many authors consider the class 
\begin{equation}
\label{eq:Phi quadratic set}
\Phi_{Q}:=\left\{ \phi:X\to\mathbb{R} \ |\  \phi\left(x\right)=a\left\Vert x\right\Vert ^{2}+\left\langle u,x\right\rangle +c \text{ where } a\leq 0 ,c\in\mathbb{R},u\in X\right\},
\end{equation}
where $X$ is a Hilbert space, see e.g. \cite[Example 6.2]{Rub2013}. For this class of elementary functions, the set $\Phi_Q$-convex functions coincides with the set of all lower semi-continuous functions defined on $X$ and minorized by a function from $\Phi_Q$, see \cite[Proposition 6.3]{Rub2013}. Quadratically minorized functions have also been investigated in \cite{attouch1993approximation}.
The starting point of numerous variant concepts of convex functions stems from the work of Hyers and Ulam \cite{ulam1952approximate}, where they defined $\delta$-convex function for $\delta\geq 0$, next approximately convex functions were investigated in e.g. \cite{daniilidis2004approximate,pales2003approximately,rolewicz2001uniformly} and the references therein. Rolewicz has coined the term "paraconvexity", \cite{rolewicz1979paraconvex}, by turning $\delta$ into a non-negative function. The definitions of strong and weak convexity were also given in \cite{vial1983strong} by Vial.

\begin{definition} \cite{Rub2013,Pall2013}
\label{def:phi-sub}
Let $X$ be a nonempty set and  $f:X\to\left( -\infty,+\infty \right]$. A $\Phi$-subgradient of $f$ at a point $x\in \text{dom } f$ is any element $\phi\in \Phi$ such that
\begin{equation}
f \left(y\right)-f \left(x\right) \geq \phi\left(y\right)-\phi \left(x\right), \quad \forall y\in X.
\end{equation}
The collection of all $\Phi$-subgradients of $f$ at $x$  is called the $\Phi$-subdifferentials of $f$ at $x$ and is denoted by  $\partial_\Phi f(x)$,
\begin{equation}
\partial_\Phi f \left(x\right):=\left\{ \phi\in \Phi\, |\,  \left(\forall y\in X\right)\ f\left(y \right) -f\left(x\right) \geq\phi\left(y\right)-\phi \left(x\right) \right\} .
\end{equation}

Furthermore, for $\varepsilon \geq 0$, the $\varepsilon-\Phi$-subdifferentials of $f$ at $x\in \text{dom } f$, $\partial_{\varepsilon,\Phi} f(x)$,  is defined as
\begin{equation}
\label{def:ephi-sub}
\partial_{\varepsilon,\Phi} f(x):=\{\phi\in\Phi\,\mid\, f \left(y\right)-f \left(x\right) \geq \phi\left(y\right)-\phi \left(x\right) -\varepsilon,\quad \forall y\in X\},
\end{equation}
an elements of  $\partial_{\varepsilon,\Phi} f(x)$ are called $(\varepsilon,\Phi)$-subgradients of $f$ at $x\in\text{dom\,}f$.
\end{definition}
Additional facts and results of $\varepsilon$-$\Phi$ subdifferentials can be found in \cite{Jey2007, burachik2008}.

The following  properties follow directly from the definitions.
\begin{proposition}
\label{prop1}
Let $X$ be a nonempty set and  $f:X\to\left( -\infty,+\infty \right]$.
\begin{enumerate}
\item For all $x\in \text{dom } f$ and $\varepsilon\geq 0$, an element $\phi\in \Phi$ is a $(\varepsilon,\Phi)$-subgradient of $f$ at $x$ if and only if 
\begin{equation}
\label{prop1:Fenchel}
f(x) + f^*_\Phi (\phi) \leq \phi (x) + \varepsilon.
\end{equation}
\item $\text{dom }f^{*} = \bigcap_{\varepsilon>0}\partial_{\varepsilon,\Phi}f\left(X\right)$.
\end{enumerate}
\end{proposition}
The proof for (i) can be found in \cite[Proposition 7.10]{Rub2013} and (ii) in \cite[Proposition 2.4]{Bui2021}.
% From now on, we always define $\Phi$ as the class of elementary functions associated with $f$ or any functions mapping from $X$ and $\Psi$ with $g$ or any functions mapping from $Y$. Whenever it is clear from the context, we  denote $\Phi$-conjugate of $f$ by $f^*$ instead of $f^*_\Phi$, and $\Phi$-subdifferentials of $f$ is denoted by $\partial f(x)$ instead of $\partial_\Phi f(x)$. We also have a similar expressions for $g^*$ and $\partial g$ instead of $g^*_\Psi$ and $\partial_\Psi g$. We specify the class of $\Phi$ when needed.

%\section{Conjugate Dual}
\section{
%Problem Formulation and Duality 
Construction of the conjugate dual}
\label{sec:Conjugate dual}

We consider the composite minimization problem
\begin{equation*}
\inf_{x\in X}f\left(x\right)+g\left(Lx\right) \tag{CP}
\end{equation*}
where $X$ is a nonempty set and $Y$ is a vector space, $f:X\to \left(-\infty,+\infty\right]$, $g:Y\to \left(-\infty,+\infty\right]$ and $L:X\to Y$ is a mapping from $X$ to $Y$. Note that here we do not assume that $L$ is linear and continuous.

Let $\Phi$ be a set of elementary functions $\phi:X\to \mathbb{R}$ and let  $\Psi$ be a set of elementary functions $\psi:Y\to \mathbb{R}$. Assume that  both classes are closed under addition of constants and $0\in \Phi,0\in\Psi$.

We introduce the conjugate dual of \eqref{prob:CP} 
by considering  perturbed minimization problems of the  form
\begin{equation}
\label{eq:beta-CP}
\inf_{x\in X} f\left(x\right)+g\left(Lx+y\right), \quad y\in Y.
\end{equation}
Following the classical  ideas coming from convex analysis, see e.g.  \cite{Bot2009,Bon2013}, we calculate the conjugate of the function $\beta: X\times Y \to (-\infty,+\infty]$, where 
$$
\beta(x,y):=f\left(x\right)+g\left(Lx+y\right),\ \ x\in X,\  y\in Y.
$$
Note that, in the convex case, conjugation is taken with respect to linear functionals, while here we are considering elementary functions $\phi,\psi$ from  given classes $\Phi$ and $\Psi$, respectively, see e.g. \cite{moreau1970}. 

Now, We define the conjugate of $\beta$ with respect to the  coupling function 
\begin{align}
c : \Phi \times\Psi \times X\times Y &\to \mathbb{R} \notag\\
\label{eq:coupling func}
\left(\phi,\psi,x,y\right) &\mapsto \phi(x) + \psi (Lx+y) - \psi(Lx).
\end{align}
For other types of coupling functions defined on Cartesian products of elementary function sets, see e.g.,  \cite{oettli1998}.  

The c-conjugate of $\beta$,   $\beta^c:\Phi\times\Psi\rightarrow(-\infty,+\infty]$, is defined as
\begin{equation}
\label{eq:beta-conj}
\beta^c (\phi,\psi) = \sup_{x\in X} \sup_{y\in Y}\left\{ c(\phi,\psi,x,y) - \beta(x,y)\right\}.
\end{equation}
The conjugate dual to problem \eqref{prob:CP} is defined as
\begin{equation}
\label{prob:DCP-form}
\sup_{\psi\in\Psi} - \beta^c (0,\psi). 
\end{equation}
Investigation of the relationship between \eqref{prob:CP} and problem \eqref{prob:DCP-form} is the main focus of the paper. To  write \eqref{prob:DCP-form} in a more explicit form, observe that the objective function of \eqref{prob:DCP-form} (the dual objective) is equal to
\begin{align}
-\beta^c (0,\psi) &= -\sup_{x\in X} \sup_{y\in Y}\left\{ c(0,\psi,x,y) - \beta(x,y)\right\} \notag\\
& = -\sup_{x\in X} \sup_{y\in Y}\left\{ \psi(Lx + y) - \psi (Lx) - f(x) - g(Lx +y)\right\} \notag\\
& = -\sup_{x\in X} \sup_{z=Lx+y }\left\{ \psi(z) - \psi (Lx) - f(x) - g(z)\right\} \notag\\
& = -\sup_{x\in X} \left\{ -\psi(Lx) - f(x) + \sup_{z\in Y} \{\psi(z) - g(z)\}\right\} \notag\\
\label{eq:beta coupling}
& = - \sup_{x\in X} \left\{-\psi (Lx) - f(x)\right\} -g^*_\Psi (\psi) \notag\\
%& = \inf_{x\in X}\left\{ \psi(Lx) + f(x) \right\}- g^*  (\psi). \notag \\
\end{align}

%where ${}^{*}f(-\psi\circ L)$ is defined in \eqref{eq:notPhi conjugate}.
Hence, the conjugate dual problem takes the form
\begin{equation}
\label{prob:DCP}
\sup_{\psi\in\Psi} - \beta^c (0,\psi) = \sup_{\psi\in\Psi} - \sup_{x\in X} \left\{-\psi (Lx) - f(x)\right\}-g^{*}_{\Psi}(\psi). \tag{DCP}
\end{equation}
Notice that we have kept the formulation of $\sup_{x\in X} \left\{-\psi (Lx) - f(x)\right\}$ in the dual problem \eqref{prob:DCP}, since we do not know if $\psi\circ L$ belongs to the class $\Phi$ in general. Below, we give some examples where we further investigate the dual problem.

\begin{example}
\label{rmk:Phi and Psi}
In some  cases, we 
%have a better representation of 
can rewrite the conjugate dual \eqref{prob:DCP} in an equivalent and more convenient way.
\begin{itemize}
\item If, for any $\psi\in\Psi$, 
%either 
$-\psi\circ L\in \Phi$ (alternatively, we can assume that $\psi\circ L\in \Phi$ and $\Phi$ is symmetric i.e. $-\phi \in \Phi$ for all $\phi\in \Phi$) 
%or $\psi\circ L\in \Phi$ and $\Phi$ is symmetric, 
we obtain
\begin{equation}
\sup\limits_{x\in X}\{-f(x)-\psi(Lx)\}
=f^*_\Phi \left(-\psi\circ L \right),
\end{equation}
and the conjugate dual problem  \eqref{prob:DCP} takes the form
\begin{equation}
\label{eq:dual-v2}
\sup_{\psi\in\Psi} - \beta^c (0,\psi) =\sup_{\psi\in\Psi} -f^*_\Phi \left(- \psi\circ L \right) -g^{*}_\Psi (\psi).
\end{equation}
%Instead of defining the set $\Phi$ in the beginning, 
%Taking into account 
With the help of the  operator $L:X\rightarrow Y $ we  define another class of elementary functions as follows
\[
\Phi_L :=\left\{-\psi\circ L: X\to \mathbb{R}, \  \psi\in\Psi 
%\text{ and } L:X\rightarrow Y \text{ is a bounded linear operator}
\right\},
\] 
and the dual  \eqref{eq:dual-v2} can be written as
\begin{equation}
\sup_{\substack{\phi\in\Phi_L,\psi\in\Psi \\ \phi +\psi\circ L =0}} -f^*_{\Phi_{L}} \left(\phi \right) -g^{*}_\Psi (\psi).
\end{equation}
\item If $\Phi$ is symmetric, and for any $\psi\in\Psi$, $\psi\circ L$ is $\Phi$-convex, we have \[
\psi (Lx) = \sup \left\{ \phi(x): \phi\in \text{supp } \psi\circ L\right\}.
\]
The optimal value of the dual problem \eqref{prob:DCP} (obtained with the help of  the coupling function $c$ from \eqref{eq:coupling func}) satisfies
\begin{align}
\sup_{\psi\in\Psi} - \beta^c (0,\psi) & =\sup_{\psi\in\Psi}\inf_{x\in X}f\left(x\right)+\psi\left(Lx\right) -g^{*}_\Psi \left(\psi\right)\notag\\
& = \sup_{\psi\in\Psi} \left\{\inf_{x\in X}\left( f(x) + \sup_{\phi\in \text{supp } \psi\circ L} \phi (x)\right) - g^*_\Psi (\psi)\right\} \notag\\
& \geq \sup_{\psi\in\Psi} \sup_{\phi\in \text{supp } \psi\circ L} \left\{\inf_{x\in X} \left(f(x) + \phi(x)\right) - g^*_\Psi (\psi) \right\} \notag\\
\label{eq:dual-v3}
& \geq \sup_{\psi\in\Psi, \phi\in \text{supp } \psi\circ L} -f^*_\Phi (-\phi) - g^*_\Psi (\psi).
\end{align}
Clearly,  the $\Phi$-convexity of $\psi\circ L$ is more general than the condition $\psi\circ L \in \Phi$. However, assuming the latter, we obtain, from \eqref{prob:DCP} i.e. 
\[
\sup_{\psi\in\Psi} - \beta^c (0,\psi) = \sup_{\psi\in\Psi} - f^*_{\Phi} (-\psi\circ L)-g^{*}_{\Psi}(\psi). 
\]
%only an inequality in \eqref{eq:dual-v3} instead of  as . 
%Moreover, the role of operator $L$ is not shown explicitly in \eqref{prob:DCP}, it is hidden in $\phi\in \text{supp } \psi\circ L$.
If $X=Y, \Phi = \Psi$ are symmetric, and $L:X\to X$ is such that $\psi\circ L\in\Phi$, for any $\psi\in\Psi$ then
\[
\sup\limits_{x\in X}\{-f(x)-\psi(Lx)\} = f^*_{\Psi} (-\psi\circ L),
\]
and the conjugate dual  \eqref{prob:DCP} takes the form 
\begin{equation}
\label{eq:dual11}
\sup_{\psi\in\Psi} - \beta^c (0,\psi) = \sup_{\psi\in\Psi} -f^*_{\Psi} (-\psi\circ L)-g^{*}_\Psi (\psi).
\end{equation}
If $L$ is the identity operator, then \eqref{prob:CP} becomes the minimization problem 
$$\inf_{x\in X} f(x) + g(x)
$$
for which the conjugate  dual \eqref{eq:dual11} has been discussed in \cite{Bed2020,Bui2021}.
\end{itemize}
\end{example}

\subsection{Conjugate dual for specific classes \texorpdfstring{$\Phi,\Psi$}{Phi, Psi}}

Now, we discuss the conjugate dual problem \eqref{prob:DCP} when $\Phi,\Psi$ are the sets of linear and quadratic functions i.e. $\Phi_{conv},\Psi_{conv}$ and $\Phi_{Q,a},\Psi_{Q,a}$ given by \eqref{eq:Phi conv set} and \eqref{eq:Phi a set} respectively. 
\begin{example}
\label{ex:3.2}
Let $X$ and $Y$ be Banach spaces with the topological duals $X^*,Y^*$ and their bilinear forms $\langle \cdot, \cdot \rangle_X, \langle \cdot, \cdot \rangle_Y$, respectively. Let $L:X\to Y$ be a linear continuous operator with the conjugate $L^* :Y^*\to X^*$. We define the sets $\Phi$ and $\Psi$ of elementary functions as follows.
\begin{enumerate}
\item For the case of affine elementary functions, (cf. \eqref{eq:Phi conv set} above), 
\begin{align*}
\Phi_{conv} & :=\left\{ \phi:X\to\mathbb{R}\ |\ \phi\left(x\right) =\left\langle u,x\right\rangle_X +c,u\in X^*, c\in \mathbb{R} \right\} \\
\Psi_{conv} & :=\left\{ \psi:Y\to\mathbb{R}\ |\ \psi\left(y\right) =\left\langle v,y\right\rangle_Y +d,v\in Y^*, d\in \mathbb{R} \right\}
\end{align*}
we have $\psi(Lx+y) - \psi (Lx) = \psi(y)$ which leads to the coupling function as defined in the convex case \cite{Bon2013}. We have
\begin{align*}
\sup\limits_{x\in X}\{-f(x)-\psi(Lx)\}
&=\sup\limits_{x\in X}\{-f(x)- \langle v,Lx\rangle_Y \}-d\\
&=\sup\limits_{x\in X}\{-f(x)- \langle L^* v,x\rangle_X \} -d. 
\end{align*} 
Since $L^* v\in X^*$, we have $-\psi \circ L\in \Phi_{conv}$ so
\[
\sup\limits_{x\in X}\{-f(x)- \langle L^* v,x\rangle_X -d \} = f^*_{\Phi_{conv}} (-\psi\circ L)
\]
and the dual problem
\begin{equation}
\label{eq:example dual 1}
\sup_{\psi\in\Psi_{conv}} - \beta^c (0,\psi) =\sup_{\psi\in\Psi_{conv}} -f^*_{\Phi_{conv}} \left(- \psi\circ L \right) -g^{*}_{\Psi_{conv}} (\psi).
\end{equation}
Note that the condition $\psi\circ L \in \Phi_{conv}$ is satisfied thanks to the form of $\Phi_{conv}$, $\Psi_{conv}$, and \eqref{eq:example dual 1} coincides with the classical Fenchel dual investigated e.g. in \cite[Chapter 1, Section 2]{Bot2009}.
\item In the case of quadratic elementary functions ($a,b\in \mathbb{R}$, cf. formula  \eqref{eq:Phi a set} above),
\begin{align}
\label{eq:class weak f}
\Phi_{Q,a} & =:\left\{ \phi:X\to\mathbb{R}\ |\ \phi\left(x\right)=a\left\Vert x\right\Vert_X ^{2}+\left\langle u,x\right\rangle_X + c,u\in X^*,c\in\mathbb{R}\right\}, \\
\label{eq:class weak g}
\Psi_{Q,b} & =:\left\{ \psi:Y\to\mathbb{R}\ |\ \psi\left(y\right)=b\left\Vert y\right\Vert_Y ^{2}+\left\langle v,y\right\rangle_Y + d,v\in Y^*,d\in\mathbb{R}\right\},
\end{align}
one cannot express $\psi\circ L$ as a function in $\Phi_{Q,a}$. However, we have
%we write $f^{*}$ with a part of $\psi\circ L$,
%consider 
\begin{align*}
\psi\left(Lx\right) &=b\left\Vert Lx\right\Vert_Y ^{2}+\left\langle v,Lx\right\rangle_Y +d\\
& =b\left\Vert x\right\Vert_{L}^{2}+\left\langle L^{*}v,x\right\rangle_X +d,
\end{align*}
with $\Vert x\Vert_L = \langle L^* L x,x\rangle_X$, \cite{Bau2011}.
Thus, the dual problem \eqref{prob:DCP} takes the form
\begin{align}
\sup_{\psi\in\Psi_{Q,b}} - \beta^c (0,\psi) &
=\sup_{\psi\in\Psi_{Q,b}}\inf_{x\in X}f\left(x\right)+\psi\left(Lx\right) -g^{*}_{\Psi_{Q,b}} \left(\psi\right)\notag\\
&=\sup_{\psi\in \Psi_{Q,b}} \inf_{x\in X}f\left(x\right)+b\Vert x\Vert_L ^{2}+\left\langle L^{*}v,x\right\rangle_X +d-g^{*}_{\Psi_{Q,b}} \left(\psi\right) \notag \\
\label{eq:example dual 2}
& =\sup_{\psi\in \Psi_{Q,b}} -\left(f+b\Vert \cdot \Vert_L ^{2}\right)^{*}_{\Phi_{Q,a}} \left( -\left\langle L^{*}v,\cdot \right\rangle_X -d\right)-g^{*}_{\Psi_{Q,b}} \left(\psi\right). 
\end{align}

% We can define $\Phi$ so that it depends on $L$ by the following
%\[
%\Phi_{L,b} :=\left\{ -\psi\circ L: X \to \mathbb{R}\ \mid \ \psi\in\Psi_b \right\}.
%\]
%Then the dual problem can be written 
%\begin{align}
%\sup_{\psi\in\Psi} - \beta^c (0,\psi) &
%=\sup_{\psi\in\Psi}\inf_{x\in X}f\left(x\right)+\psi\left(Lx\right) -g^{*}\left(\psi\right)\notag\\
%\label{eq:example dual 3}
%& =\sup_{\psi\in\Psi_b} -f^{*} \left(-\psi\circ L\right)-g^{*}\left(\psi\right).
%\end{align}

\item In the construction of dual problem \eqref{prob:DCP-form}, the roles of $\Phi$ and $\Psi$ are not symmetric. We can see this by considering the following pair of elementary functions,
\begin{align*}
\Phi_{Q,a} & :=\left\{ \phi:X\to\mathbb{R}\ |\ \phi\left(x\right)=a\left\Vert x\right\Vert_X ^{2}+\left\langle u,x\right\rangle_X + c ,u\in X^{*},c\in\mathbb{R}\right\}, \\
\Psi_{conv} & :=\left\{ \psi:Y\to\mathbb{R}\ |\ \psi\left(y\right)= \left\langle v,y\right\rangle_Y + d ,v\in Y^{*}, d\in\mathbb{R} \right\}.
\end{align*}
In this case, since $\psi\in\Psi_{conv}$ is an affine function we have 
\[
-\psi \circ L (\cdot) = -\langle  v, L\cdot \rangle_Y -d =-\langle L^* v,\cdot\rangle_X -d :=\phi (\cdot)   \in \Phi_{Q,a}
\]
and the dual problem has the same form as \eqref{eq:example dual 1} 
\begin{align}
\sup_{\psi\in\Psi_{conv}} - \beta^c (0,\psi) & =\sup_{\psi\in\Psi_{conv}} -f^*_{\Phi_{Q,a}} \left(- \psi\circ L \right) -g^{*}_{\Psi_{conv}} (\psi) 
\end{align}

However, when reversing the roles of $\Phi$ and $\Psi$ i.e. by taking
\begin{align*}
\Phi_{conv} & :=\left\{ \phi:X\to\mathbb{R}|\ \phi\left(x\right)=\left\langle u,x\right\rangle_X + c ,u\in X^{*}, c\in \mathbb{R}\right\} \\
\Psi_{Q,b} & :=\left\{ \psi:Y\to\mathbb{R}|\ \psi\left(y\right)= b\left\Vert y\right\Vert ^{2}_Y+\left\langle v,y\right\rangle_Y +d,v\in Y^{*},b,d\in\mathbb{R}\right\}.
\end{align*}
we cannot write the dual problem in the form \eqref{eq:example dual 1}, but it is possible to write, for any $\psi\in\Psi_{Q,b}$, 
\[
\psi(Lx) = b\lVert Lx\rVert^2_Y +\langle v,Lx\rangle_Y +d=\psi_1 (x) + \psi_2 (x),
\]
where $\psi_1 (x) = b\lVert Lx\rVert^2_Y,\psi_2 (x) = \langle L^*v,x\rangle_X +d$ and the dual problem takes the form
\begin{equation}
\sup_{\substack{\psi\in\Psi_{Q,b}\\ \psi\circ L = \psi_1 +\psi_2}} - \beta^c (0,\psi) 
=\sup_{\substack{\psi\in \Psi_{Q,b}\\ \psi\circ L = \psi_1 +\psi_2}} -\left(f+\psi_1 \right)^{*}_{\Phi_A} \left( -\psi_2\right)-g^{*}_{\Psi_{Q,b}} \left(\psi\right).
\end{equation}
\end{enumerate}
\end{example}
Unlike the dual problems, considered in \cite{Bed2020,Bui2021} where $L=Id$, the appearance of general operators $L$ in the minimization problem \eqref{prob:CP} makes it more difficult to present the dual problem without imposing additional  assumptions.
\begin{remark}
\label{rmk:const}
It is clear from the formula \eqref{eq:beta coupling} and the above examples that the presence of constants in the definition of elementary functions is inessential in the $c$-conjugate of $\beta$. Thus, in the sequel, we consider classes $\Phi$ and $\Psi$ of functions in which constants are omitted. For the general consideration related to this fact, see \cite[formula 1.4.4]{Rub2013}.
\end{remark}

\section{Zero Duality Gap for Conjugate Dual}
\label{sec:zero dual gap}
Having defined the 
%primal \eqref{prob:CP} and 
dual problem \eqref{prob:DCP}, now, we discuss conditions for weak duality and zero duality gap. 
As in \eqref{prob:CP}, we assume that the sets $\Phi$ and $\Psi$ of elementary functions are defined on $X$ and $Y$, respectively, and both contain zeros, i.e., $0\in\Phi,0\in\Psi$. 

Let $X$ be a nonempty set and let $Y$ be a vector space. Let $L:X\to Y$ be a mapping from $X$ to $Y$. The dual problem \eqref{prob:DCP} can be equivalently rewritten in the form
\begin{align}
\label{prob:DCP-general}
val (DCP) :=\sup_{\psi\in\Psi} - \beta^c (0,\psi)  &=-\inf_{\psi\in\Psi}\left\{ \left(f+\psi\circ L\right)^{*}_\Phi \left(0\right)+g^{*}_\Psi \left(\psi\right) \right\}
\end{align}
%The primal problem \eqref{prob:CP} can  be written in the conjugate form
For the problem \eqref{prob:CP}, we have
\begin{equation}
val (CP) :=\inf_{x\in X}f\left(x\right)+g\left(Lx\right)=-\left(f+g\circ L\right)^{*}_\Phi\left(0\right).
\end{equation}
Weak duality, i.e. the inequality  $val (CP)\geq val (DCP)$ 
means that
\begin{equation}
\label{eq:weak duality}
\left(f+g\circ L\right)^{*}_\Phi\left( 0 \right) \leq\inf_{\psi\in\Psi}\left(f+\psi\circ L\right)^{*}_\Phi\left( 0 \right) +g^{*}_\Psi \left(\psi\right).
\end{equation}
In fact, a more general inequality holds true.

\begin{theorem}
\label{thm:CP weak dual}
For any $\phi\in\Phi$, it holds
\[
\left(f+g\circ L\right)^{*}_\Phi \left(\phi\right)\leq\inf_{\psi\in\Psi}\left(f+\psi\circ L\right)^{*}_\Phi\left(\phi\right)+g^{*}_\Psi \left(\psi\right).
\]
\end{theorem}

\begin{proof}
For any $\psi\in \Psi$ and $\phi\in\Phi$, we have
\begin{align*}
\left(f+g\circ L\right)^{*}_\Phi \left(\phi\right) & =\sup_{x\in X}\phi\left(x\right)-f\left(x\right)-g\left(Lx\right)\\
 & =\sup_{x\in X}\phi\left(x\right)-f\left(x\right)-g\left(Lx\right)+\psi\left(Lx\right)-\psi\left(Lx\right)\\
 & \leq\sup_{x\in X} \left\{\phi\left(x\right)-f\left(x\right)-\psi\left(Lx\right) \right\}+\sup_{x\in X} \left\{\psi\left(Lx\right)-g\left(Lx\right)\right\}\\
 & \leq\left(f+\psi\circ L\right)^{*}_\Phi  \left(\phi\right) + \sup_{y\in L(X)} \left\{\psi\left(y\right) - g\left(y\right)\right\}\\
 & \leq\left(f+\psi\circ L\right)^{*}_\Phi \left(\phi\right)+g^{*}_\Psi \left(\psi\right).
\end{align*}

Since $\left(f+g\circ L\right)^{*}_\Phi \left(\phi\right)$ is a lower bound
of $\left(f+\psi\circ L\right)^{*}_\Phi \left(\phi\right)+g^{*}_\Psi \left(\psi\right)$
for arbitrary $\psi\in\Psi$, we get
\[
\left(f+g\circ L\right)^{*}_\Phi \left(\phi\right)\leq\inf_{\psi\in\Psi}\left(f+\psi\circ L\right)^{*}_\Phi \left(\phi\right)+g^{*}_\Psi \left(\psi\right),
\]
which completes the proof.
\end{proof}
By taking $\phi = 0$ in the above theorem, we obtain the weak duality \eqref{eq:weak duality} between \eqref{prob:CP} and \eqref{prob:DCP}. Note that no additional assumptions are imposed  on the functions $\psi\circ L$ and $\phi$. However, when zero duality gap is considered, we need to assume a relationship  between $\Phi$ and $\Psi$.

Our first result of zero duality gap is related to \cite[Theorem 3.5]{Bui2021}.
%We give two versions of the result which follows  \cite[Theorem 3.5]{Bui2021}. Let us state the first one.

\begin{theorem}
\label{thm:CP zero gap}
Let $X$ be a nonempty set and $Y$ be a vector space.
Let $f:X\to\left(-\infty,+\infty\right]$, $g:Y\to\left(-\infty,+\infty\right]$, and $L:X\to Y$ be a mapping. Assume that $\text{dom } g\cap L\left(\text{dom }f\right)\neq\emptyset$.
Suppose $0\in\Phi$ and $0\in\Psi$. The following are equivalent.
\begin{enumerate}
\item For every $\varepsilon>0$, there exist $x_{\varepsilon}\in X, \psi_\varepsilon\in\partial_{\varepsilon,\Psi} g(Lx_\varepsilon)$ such that 
$\left(f+\psi_\varepsilon\circ L\right)^*_\Phi \left(0\right) \leq -f(x_\varepsilon) - \psi_\varepsilon\left(Lx_\varepsilon\right) +\varepsilon$.
\item $-val (CP)=\left(f+g\circ L\right)^{*}_\Phi\left(0\right)=\inf_{\psi\in\Psi}\left(f+\psi\circ L\right)^{*}_\Phi\left(0\right)+g^{*}_\Psi \left(\psi\right) = - val (DCP) <+\infty$.
\end{enumerate}
\end{theorem}
\begin{proof}
(i) $\Rightarrow$ (ii): Thanks to Theorem \ref{thm:CP weak dual}, we only need to prove that 
\[
\inf_{\psi\in\Psi}\left(f+\psi\circ L\right)^{*}_\Phi\left(0\right)+g^{*}_\Psi \left(\psi\right)\leq\left(f+g\circ L\right)^{*}_\Phi\left(0\right)<+\infty.
\]

Let $\varepsilon>0$. From (i), there exist $x_{\varepsilon}\in X$ and $\psi_\varepsilon \in \partial_{\varepsilon, \Psi} g\left(Lx_{\varepsilon}\right)$ such that
\[
\left(f+\psi_\varepsilon\circ L\right)^*_\Phi \left(0\right) \leq -f(x_\varepsilon) - \psi_\varepsilon\left(Lx_\varepsilon\right) +\varepsilon.
\]
From the definition of infimum, we have
\begin{align}
\inf_{\psi\in\Psi}\left(f+\psi\circ L\right)^{*}_\Phi\left(0\right)+g^{*}_\Psi \left(\psi\right) & \leq\left(f+\psi_{\varepsilon}\circ L\right)^{*}_\Phi\left(0\right)+g^{*}_\Psi \left(\psi_{\varepsilon}\right) \notag\\
\label{eq:-f-psiL}
 & \leq-f\left(x_{\varepsilon}\right)-\psi_{\varepsilon}\left(Lx_{\varepsilon}\right)+g^*_\Psi \left(\psi_{\varepsilon}\right)+\varepsilon.
\end{align}
By using inequality \eqref{prop1:Fenchel} applied to $g$ and taking into account $\psi_\varepsilon \in \partial_{\varepsilon,\Psi} g(Lx_\varepsilon)$, we obtain
\begin{align}
\label{eq:-f-gL}
-f\left(x_{\varepsilon}\right)-\psi_{\varepsilon}\left(Lx_{\varepsilon}\right)+g^*_\Psi \left(\psi_{\varepsilon}\right)+\varepsilon & \leq-f\left(x_{\varepsilon}\right)-g\left(Lx_{\varepsilon}\right)+2\varepsilon\\
 & \leq\left(f+g\circ L\right)^{*}_\Phi\left(0\right)+2\varepsilon.\notag
\end{align}
Finally, by using \eqref{eq:-f-psiL} and \eqref{eq:-f-gL},
\[
\inf_{\psi\in\Psi}\left(f+\psi\circ L\right)^{*}_\Phi\left(0\right)+g^{*}_\Psi \left(\psi\right) \leq \left(f+g\circ L\right)^{*}_\Phi\left(0\right)+2\varepsilon.
\]
As both sides of the above inequality do not depend on $x_{\varepsilon}$ or $\psi_{\varepsilon}$, we can let $\varepsilon\to0$ and obtain zero duality gap 
\[
\inf_{\psi\in\Psi}\left(f+\psi\circ L\right)^{*}_\Phi\left(0\right)+g^{*}_\Psi\left(\psi\right)\leq\left(f+g\circ L\right)^{*}_\Phi\left(0\right).
\]
%Lastly, notice that 
%\begin{align*}
%\left(f+g\circ L\right)_{\Phi}^{*}\left(0\right)=\sup_{x\in X}-f\left(x\right)-g\left(Lx\right)	& \leq\sup_{x\in X}\left\{ -f\left(x\right)-\psi_{\varepsilon}\left(Lx\right)\right\} +\sup_{x\in X}\left\{ \psi_{\varepsilon}\left(Lx\right)-g\left(Lx\right)\right\} \\
%	& \leq\left(f+\psi_{\varepsilon}\circ L\right)_{\Phi}^{*}\left(0\right)+g_{\Psi}^{*}\left(\psi_{\varepsilon}\right).
%\end{align*}
From \eqref{eq:-f-psiL} and \eqref{eq:-f-gL}, we have 
\[
\left(f+g\circ L\right)_{\Phi}^{*}\left(0\right) \leq -f\left(x_{\varepsilon}\right)-g\left(Lx_{\varepsilon}\right)+2\varepsilon <+\infty,
\]
from the assumption of $f$ and $g$. Thus $\left(f+g\circ L\right)_{\Phi}^{*}\left(0\right)<+\infty$.

(ii) $\Rightarrow$ (i):
From the definition of supremum, for every $\varepsilon>0$
there exists an $x_{\varepsilon} \in X$ such that 
\[
\left(f+g\circ L\right)^{*}_\Phi\left(0\right)\leq-f\left(x_{\varepsilon}\right)-g\left(Lx_{\varepsilon}\right)+\varepsilon/2.
\]
 Also, there  exists $\psi_{\varepsilon}\in\Psi$ such that 
\[
\left(f+\psi_{\varepsilon}\circ L\right)^{*}_\Phi\left(0\right)+g^{*}_\Psi \left(\psi_{\varepsilon}\right)-\varepsilon/2<\inf_{\psi\in\Psi}\left(f+\psi\circ L\right)^{*}_\Phi\left(0\right)+g^{*}_\Psi \left(\psi\right).
\]
Zero duality gap gives us 
\[
\left(f+\psi_{\varepsilon}\circ L\right)^{*}_\Phi\left(0\right)+g^{*}_\Psi \left(\psi_{\varepsilon}\right)-\varepsilon/2\leq-f\left(x_{\varepsilon}\right)-g\left(Lx_{\varepsilon}\right)+\varepsilon/2.
\]
After rearranging both sides, we get
\[
\left[\left(f+\psi_{\varepsilon}\circ L\right)^{*}_\Phi\left(0\right)+f\left(x_{\varepsilon}\right)+\psi_{\varepsilon}\left(Lx_{\varepsilon}\right)\right]+\left[g^{*}_\Psi \left(\psi_{\varepsilon}\right)+g\left( Lx_{\varepsilon} \right)-\psi_{\varepsilon}\left(Lx_{\varepsilon}\right)\right]<\varepsilon.
\]
By the definition of conjugate function, each of the two terms is non-negative,
and so each of them has to be smaller than $\varepsilon$. Thus
\begin{align*}
\left(f+\psi_{\varepsilon}\circ L\right)^{*}_\Phi\left(0\right) & <-f\left(x_{\varepsilon}\right)-\psi_{\varepsilon}\left(Lx_{\varepsilon}\right)+ \varepsilon\\
g^{*}_\Psi\left(\psi_{\varepsilon}\right)+g\left( Lx_{\varepsilon}\right)-\psi_{\varepsilon}\left(Lx_{\varepsilon}\right) & <\varepsilon,
\end{align*}
which implies that  $\psi_{\varepsilon}\in\partial_{\varepsilon, \Psi} g\left( Lx_{\varepsilon}\right)$.
Hence, (i) holds.
\end{proof}

\begin{remark}
Observe that, in some cases, we cannot find $\phi\in \Phi,\psi\in\Psi$ such that $\phi+\psi=0$ (this latter condition is used in \cite[Theorem 3.5]{Bui2021}). 
In the case $X=Y, \Phi = \Psi, L=Id$ and $\Phi$ is symmetric (for any $\psi \in \Phi, -\psi \in \Phi$), Theorem \eqref{thm:CP zero gap}-(i) means that, for $\psi_\varepsilon\in\partial_{\varepsilon,\Phi} g(x_\varepsilon)$,
\[
-f(x) - \psi_\varepsilon(x) \leq \left(f+\psi_\varepsilon \right)^*_\Phi \left(0\right) \leq -f(x_\varepsilon) - \psi_\varepsilon\left(x_\varepsilon\right) +\varepsilon,
\]
i.e. $-\psi_\varepsilon \in \partial_{\varepsilon,\Phi} f (x_\varepsilon)$. This means that $0\in \bigcap_{\varepsilon>0} \partial_{\varepsilon,\Phi} (f + g)(X)$, which reduces to the respective condition used in  \cite[Theorem 3.5-(i)]{Bui2021} for proving zero duality gap.
\end{remark}

Condition (i) of Theorem \ref{thm:CP zero gap}, could be replaced by conditions which are easier to be checked, when we introduce the following assumption: for given $\phi\in\Phi,\psi\in\Psi$, 
\begin{align}
\label{eq:CP cond 1}
\phi - \psi\circ L \in \Phi. 
\end{align}
Note that condition \eqref{eq:zero gap system} below will be used in the sequel to check zero duality gap.
\begin{remark}
When $\Phi$ is a convex cone, then \eqref{eq:CP cond 1} is satisfied when $-\psi\circ L \in \text{rec } \Phi$ where $\text{rec }\Phi = \left\{ \varphi\in F(X): \varphi + \Phi \subset \Phi\right\}$ is the recession cone of $\Phi$ and $F(X)$ is a linear space of all functions defined on $X$.
\end{remark}
\begin{theorem}
\label{thm:CP zero gap v2}
Let $f:X\to\left(-\infty,+\infty\right], g:Y\to\left(-\infty,+\infty\right]$ be such that $\text{dom }g\cap L\left(\text{dom }f\right)\neq\emptyset$. Let $L:X\to Y$ be a mapping from $X$ into $Y$. Suppose that $0\in\Phi$ and $0\in \Psi$. Consider the following conditions.
\begin{enumerate}
\item For every $\varepsilon>0$, there exist $x_\varepsilon\in X,\phi_\varepsilon \in\partial_{\varepsilon, \Phi} f (x_\varepsilon), \psi_\varepsilon \in \partial_{\varepsilon, \Psi}g \left(Lx_\varepsilon\right)$ such that

\begin{equation}
\label{eq:zero gap system}
\begin{cases}
\phi_\varepsilon (z) + \psi_\varepsilon (Lz) \geq - \varepsilon & \forall z\in X\\
\phi_\varepsilon(x_\varepsilon) +\psi_\varepsilon(Lx_\varepsilon)\leq \varepsilon.
\end{cases}
\end{equation} 
\item $-val(CP) = \left(f+g\circ L\right)^{*}_\Phi\left(0\right)=\inf_{\psi\in\Psi}\left(f+\psi\circ L\right)^{*}_\Phi\left(0\right)+g^{*}_\Psi\left(\psi\right) = -val (DCP)<+\infty$. 
\end{enumerate}
We have (i) $\Rightarrow$ (ii). If, for every $\varepsilon >0$, there exists $\psi_\varepsilon \in \Psi$ such that
\[
\left(f+\psi_{\varepsilon}\circ L\right)^{*}_\Phi\left(0\right)+g^{*}_\Psi\left(\psi_{\varepsilon}\right)-\varepsilon/2 <\inf_{\psi\in\Psi}\left(f+\psi\circ L\right)^{*}_\Phi\left(0\right)+g^{*}_\Psi \left(\psi\right)
\]
and $-\psi_\varepsilon\circ L \in \Phi$, then (i)$\Leftrightarrow$(ii).
\end{theorem}

\begin{proof}
(i) $\Rightarrow$ (ii). For any $\psi \in \Psi$
\[
\inf_{\psi\in\Psi}\left(\left(f+\psi\circ L\right)^{*}_\Phi\left(0\right)+g^{*}_\Psi \left(\psi\right)\right)\leq\left(f+\psi\circ L\right)^{*}_\Phi\left(0\right)+g^{*}_\Psi \left(\psi\right).
\]
By using inequality \eqref{prop1:Fenchel}, for $\psi_\varepsilon \in\partial_{\varepsilon, \Psi} g(Lx_\varepsilon), \phi_\varepsilon\in \partial_{\varepsilon, \Phi} f(x_\varepsilon)$
\begin{align*}
\left(f+\psi_\varepsilon\circ L\right)^{*}_\Phi\left(0\right)+g^{*}_\Psi\left(\psi_\varepsilon\right) & \leq\left(f+\psi_\varepsilon\circ L\right)^{*}_\Phi\left(0\right)+\psi_\varepsilon\left(Lx_\varepsilon\right)-g\left(Lx_\varepsilon\right)+\varepsilon\\
 & =\sup_{x\in X}\left\{ -\psi_\varepsilon\left(Lx\right)-f\left(x\right)\right\} +\psi_\varepsilon\left(Lx_\varepsilon\right)-g\left(Lx_\varepsilon\right)+\varepsilon\\
 & \leq f^{*}_\Phi \left(\phi_\varepsilon\right)+\psi_\varepsilon\left(Lx_\varepsilon\right)-g\left(Lx_\varepsilon\right)+2\varepsilon\\
 & \leq\phi_\varepsilon\left(x_\varepsilon\right)-f\left(x_\varepsilon\right)+\varepsilon+\psi_\varepsilon\left(Lx_\varepsilon\right)-g\left(Lx_\varepsilon\right)+2\varepsilon\\
 & \leq3\varepsilon-f\left(x_\varepsilon\right)-g\left(Lx_\varepsilon\right)\\
 & \leq3\varepsilon+\left(f+g\circ L\right)^{*}_\Phi\left(0\right),
\end{align*}
where in the third estimation, we use \eqref{eq:zero gap system} to obtain $f^*_\Phi (\phi_\varepsilon)$.

By letting $\varepsilon\to0$ we obtain zero duality gap. We can use the same argument as in  Theorem \ref{thm:CP zero gap} to prove that 
\[
\left(f+g\circ L\right)^{*}_\Phi\left(0\right)<+\infty.
\]

(ii) $\Rightarrow$ (i). We have 
\[
\inf_{\psi\in\Psi}\left(\left(f+\psi\circ L\right)^{*}_\Phi\left(0\right)+g^{*}_\Psi \left(\psi\right)\right)=\left(f+g\circ L\right)^{*}_\Phi\left(0\right)<+\infty,
\]
so for every $\varepsilon > 0$, there exists  $\psi_{\varepsilon}\in\Psi$ such that 
\[
\left(f+\psi_{\varepsilon}\circ L\right)^{*}_\Phi\left(0\right)+g^{*}_\Psi \left(\psi_{\varepsilon}\right)-\varepsilon/2 <\left(f+g\circ L\right)^{*}_\Phi\left(0\right),
\]
and $-\psi_\varepsilon \circ L\in\Phi$.
Moreover, there exists an $x_\varepsilon\in X$ such that 
\[
\left(f+g\circ L\right)^{*}_\Phi\left(0\right)<\varepsilon/2 -f\left(x_\varepsilon\right)-g\left(Lx_\varepsilon\right).
\]
Combining the two inequalities gives us
\[
\left(f+\psi_{\varepsilon}\circ L\right)^{*}_\Phi\left(0\right)+g^{*}_\Psi \left(\psi_{\varepsilon}\right)+f\left(x_\varepsilon\right)+g\left(Lx_\varepsilon\right)<\varepsilon.
\]
 By adding and substracting $\psi_{\varepsilon}\left(Lx_\varepsilon\right)$, we get
\[
\left[\left(f+\psi_{\varepsilon}\circ L\right)^{*}_\Phi\left(0\right)+f\left(x_\varepsilon\right)+\psi_{\varepsilon}\left(Lx_\varepsilon\right)\right]+\left[g^{*}_\Psi \left(\psi_{\varepsilon}\right)+g\left(Lx_\varepsilon\right)-\psi_{\varepsilon}\left(Lx_\varepsilon\right)\right]<\varepsilon.
\]

Since each term is nonnegative, we obtain 
\begin{align*}
\left(f+\psi_{\varepsilon}\circ L\right)^{*}_\Phi\left(0\right)+f\left(x_\varepsilon\right)+\psi_{\varepsilon}\left(Lx_\varepsilon\right) & \leq\varepsilon\\
g^{*}_\Psi\left(\psi_{\varepsilon}\right)+g\left(Lx_\varepsilon\right)-\psi_{\varepsilon}\left(Lx_\varepsilon\right) & \leq\varepsilon.
\end{align*}

Since $-\psi_{\varepsilon}\circ L \in \Phi$, and thus $\left(f+\psi_{\varepsilon}\circ L\right)^{*}_\Phi\left(0\right)=  f^{*}_\Phi \left(-\psi_\varepsilon \circ L\right)$. By the above inequalities,
$\psi_{\varepsilon}\in\partial_{\varepsilon, \Psi} g\left(Lx_{\varepsilon}\right)$
and 
\[
\left(f+\psi_{\varepsilon}\circ L\right)^{*}_\Phi\left(0\right)+f\left(x_\varepsilon\right)+\psi_{\varepsilon}\left(Lx_\varepsilon\right)=f^{*}_\Phi \left(-\psi_\varepsilon \circ L \right)+f\left(x_\varepsilon\right)+ \psi_{\varepsilon} (L x_\varepsilon) \leq\varepsilon
\]
which is equivalent to $-\psi_\varepsilon \circ L\in\partial_{\varepsilon, \Phi} f\left(x_\varepsilon\right)$.
Hence, (i) is proved.
\end{proof}
%{\color{blue} Delete Remark 3 as it seems unnecessary}
%\begin{remark}
%Statement (i) from Corollary \ref{thm:CP zero gap v2} reduces to  \cite[Theorem 3.5-(i)]{Bui2021} in the case of two functions, $L=Id$, and $\Phi=\Psi$. In addition, we need to assume \eqref{eq:CP cond 1} so that condition $\psi\circ L +\phi =0$ holds for $\phi\in\Phi,\psi\in \Psi$.
%\end{remark}

% add another remark

%\begin{remark}
%As we notice in Theorem 4.3-(ii): We can ensure the existence of $\psi_\varepsilon\in\Psi,x_\varepsilon\in X$ for every $\varepsilon>0$.
%By replacing condition \eqref{eq:CP cond 1} with a milder condition i.e. there exists $\phi_\varepsilon\in \Phi$ such that
%\begin{equation}
%\label{eq:zero gap system reverse}
%\begin{cases}
%\left(f+\psi_{\varepsilon}\circ L\right)^{*}_\Phi\left(0\right)+g^{*}_\Psi\left(\psi_{\varepsilon}\right)-\varepsilon/2 <\inf_{\psi\in\Psi}\left(f+\psi\circ L\right)^{*}_\Phi\left(0\right)+g^{*}_\Psi \left(\psi\right)\\
%\phi_\varepsilon (z) + \psi_\varepsilon (Lz) \leq - \varepsilon/2, \quad \forall z\in X\\
%\phi_\varepsilon(x_\varepsilon) +\psi_\varepsilon(Lx_\varepsilon)\geq \varepsilon/2,
%\end{cases}
%\end{equation}
%then $\phi_\varepsilon \in\partial_{\varepsilon, \Phi} f (x_\varepsilon), \psi_\varepsilon \in \partial_{\varepsilon, \Psi}g \left(Lx_\varepsilon\right)$. Unfortunately we cannot guarantee that \eqref{eq:zero gap system} holds.
%\end{remark}}
Next we obtain the existence of optimal solution to \eqref{prob:CP} in the spirit of \cite[Theorem 3.6]{Bui2021}.

\begin{theorem}
\label{thm:CP-strong dual}
Let $f:X\to\left(-\infty,+\infty\right],g:Y\to\left(-\infty,+\infty\right]$ be such that $\text{dom }g\cap L\left(\text{dom }f\right)\neq\emptyset$ and let $L:X\to Y$ be a mapping from $X$ into $Y$. Suppose that $0\in\Phi,0\in\Psi$ and $x\in \text{dom }g\cap L\left(\text{dom }f\right) $. Consider the following conditions.
\begin{enumerate}
\item For all $\varepsilon>0$, there exist $\phi_\varepsilon \in \partial_{\varepsilon, \Phi} f(x),\psi_\varepsilon \in \partial_{\varepsilon, \Psi} g\left(Lx\right)$ such that 
\begin{align*}
\begin{cases}
\phi_\varepsilon (z) + \psi_\varepsilon (Lz) \geq -\varepsilon \quad \text{for all } z\in X \\
\phi_\varepsilon (x)+\psi_\varepsilon (Lx)\leq \varepsilon.
\end{cases}
\end{align*}

\item $-val(CP) = \left(f+g\circ L\right)^{*}_\Phi\left(0\right)=\inf_{\psi\in\Psi}\left(f+\psi\circ L\right)^{*}_\Phi\left(0\right)+g^{*}_\Psi\left(\psi\right)= val (DCP)<+\infty$ and $x$ is an optimal solution to \eqref{prob:CP}. 
\end{enumerate}
We have (i) $\Rightarrow$ (ii). For every $\varepsilon>0$, if there exists $\psi_\varepsilon \in \Psi, \left(f+\psi_{\varepsilon}\circ L\right)^{*}_\Phi\left(0\right)+g^{*}_\Psi\left(\psi_{\varepsilon}\right)-\varepsilon/2 <\inf_{\psi\in\Psi}\left(f+\psi\circ L\right)^{*}_\Phi\left(0\right)+g^{*}_\Psi\left(\psi\right)$ such that $-\psi_\varepsilon\circ L \in \Phi$, then the two statements are equivalent.
\end{theorem}

\begin{proof}
The lines of the proof coincide with the ones of Theorem \ref{thm:CP zero gap v2}. We only need to prove that $x$ is an optimal solution to \eqref{prob:CP}. Now as (i) holds for all $\varepsilon>0$, we have 
\begin{align*}
\inf_{\psi\in\Psi}\left(\left(f+\psi\circ L\right)^{*}_\Phi\left(0\right)+g^{*}_\Psi\left(\psi\right)\right)& \leq\left(f+\psi_\varepsilon\circ L\right)^{*}_\Phi\left(0\right)+g^{*}_\Psi\left(\psi_\varepsilon\right)\\
& \leq3\varepsilon-f\left(x\right)-g\left(Lx\right) \\
& \leq 3\varepsilon + \left(f+g\circ L\right)^*_\Phi (0).
\end{align*}
Letting $\varepsilon\to 0$ in the above inequality as it holds for all $\varepsilon>0$. We obtain zero duality gap
\[
\inf_{\psi\in\Psi}\left(\left(f+\psi\circ L\right)^{*}_\Phi\left(0\right)+g^{*}_\Psi\left(\psi\right)\right) = (f+g\circ L)^*_\Phi (0).
\]
Notice that
\[
\inf_{z\in X} f(z)+ g(Lz) = -(f+g\circ L)^*_\Phi (0) \geq f\left(x\right) + g\left(Lx\right).
\]
As $(f+g\circ L)^*_\Phi (0) <+\infty$, the primal problem \eqref{prob:CP} is finite with $x$ as an optimal solution.
\end{proof}
\begin{remark}\footnote{We thank an annonymuous referee for this observation} 
Let us observe that the condition (i) of Theorem 4.3  and Theorem 4.4 can be replaced by the condition 
\begin{equation}
\label{eq: phi + psi L =0}
\phi_\varepsilon (x) +\psi_\varepsilon (Lx) = 0, \quad \forall x\in X.
\end{equation}
Indeed, if \eqref{eq: phi + psi L =0} holds, which means $-\psi_\varepsilon \circ L = \phi_\varepsilon \in \Phi$, then \eqref{eq:zero gap system} and \eqref{eq:CP cond 1} are satisfied.
\end{remark}

\subsection{Zero duality gap for specific classes \texorpdfstring{$\Phi,\ \Psi$}{Phi,Psi}}

%If $f$ and $g$ are weakly convex functions, we have a
Consider the classes of elementary functions as in \eqref{eq:Phi a set}. i.e. $\Phi_{Q,a}$ and $\Psi_{Q,b}$ defined in \eqref{eq:class weak f} and \eqref{eq:class weak g} for $a,b\leq 0$. We let $c=d=0$ because they do not affect the dual problem and the condition for zero duality gap. The dual problem \eqref{eq:example dual 2} takes the form 
\begin{equation}
 \sup_{\substack{\psi\in \Psi_{Q,b} \\ \psi (\cdot)= -b\Vert \cdot \Vert^2_Y + \langle v, \cdot \rangle_Y} } -\left(f+ b\left\Vert L \cdot \right\Vert _{Y}^{2}\right)^{*}_{\Phi_{Q,a}} \left( -\langle L^* v, \cdot \rangle\right)-g^{*}_{\Psi_{Q,b}} \left(\psi\right).
\end{equation}
\begin{proposition}
Let $X$ and $Y$ be Hilbert spaces with inner products, $\langle\cdot,\cdot\rangle_{X}$ and $\langle\cdot,\cdot\rangle_{Y}$, respectively. Let $L:X\rightarrow Y$ be a continuous linear operator from $X$ to $Y$, and $f:X\rightarrow(-\infty,+\infty]$ be a weakly convex function with modulus $a$. Then $f(x)-b\lVert Lx\rVert^2_Y$ is weakly convex with modulus $a+b\lVert L\rVert^2_{\mathcal{L}(X,Y)}$.
\end{proposition}
\begin{proof}
By Definition \ref{def:weak convex}, 
a function $f:X\rightarrow(-\infty,+\infty]$ is weakly convex on $X$ with modulus $a>0$ if $f+a\lVert x\rVert^{2}_{X}$, is a convex function, or equivalently for any $x_1,x_2 \in X$ and $\lambda\in [0,1]$,
$$
f(\lambda x_{1}+(1-\lambda)x_{2})\le \lambda f(x_{1})+(1-\lambda)f(x_{2})+a\lambda(1-\lambda)\|x_{1}-x_{2}\|_{X}^{2}.
$$
Now we analyze the weak convexity of the function $f(x)-b\lVert Lx\rVert^2_Y$, where $b\geq 0$. We have
\begin{align*}
\lVert L(\lambda x_1 + (1-\lambda)x_2)\rVert^2_Y & = \lVert \lambda Lx_1 +(1-\lambda)Lx_2\rVert^2_Y \\ 
& = \lambda^2 \lVert Lx_1 \rVert^2_Y + (1-\lambda)^2 \lVert Lx_2\rVert^2_Y + 2\lambda (1-\lambda)\langle Lx_1,Lx_2 \rangle_Y \\
&=\lambda \lVert Lx_1 \rVert^2_Y + (1-\lambda) \lVert Lx_2\rVert^2_Y - \lambda (1-\lambda)\lVert Lx_1 - Lx_2 \rVert^2_Y.
\end{align*}
Denote $x_{\lambda}:=\lambda x_{1}+\left(1-\lambda\right)x_{2}$ for $   \lambda\in\left[0,1\right]$. 
Therefore,
\begin{align*}
f\left(x_{\lambda}\right)-b\left\Vert Lx_{\lambda}\right\Vert _{Y}^{2}&\leq\lambda\left(f\left(x_{1}\right)-b\left\Vert Lx_{1}\right\Vert _{Y}^{2}\right)+\left(1-\lambda\right)\left(f\left(x_{2}\right)-b\left\Vert Lx_{2}\right\Vert _{Y}^{2}\right)\\
&+b\lambda\left(1-\lambda\right)\left\Vert Lx_{1}-Lx_{2}\right\Vert _{Y}^{2}+a\lambda\left(1-\lambda\right)\left\Vert x_{1}-x_{2}\right\Vert _{X}^{2}.
\end{align*}
As $L$ is linear and continuous, we obtain 
\begin{align*}
f\left(x_{\lambda}\right)-b\left\Vert Lx_{\lambda}\right\Vert _{Y}^{2}&\leq\lambda\left(f\left(x_{1}\right)-b\left\Vert Lx_{1}\right\Vert _{Y}^{2}\right)+\left(1-\lambda\right)\left(f\left(x_{2}\right)-b\left\Vert Lx_{2}\right\Vert _{Y}^{2}\right)\\
&+\left(a+b \lVert L\rVert^2_{\mathcal{L}(X,Y)} \right)\lambda\left(1-\lambda\right)\left\Vert x_{1}-x_{2}\right\Vert _{X}^{2},
\end{align*}
where $\lVert L\rVert_{\mathcal{L}(X,Y)} = \sup_{\Vert x\Vert_X \leq 1} \Vert Lx \Vert_Y$.
This means $f(x)-b\lVert Lx\rVert^2_Y$ is weakly convex with modulus $a+b\lVert L\rVert^2_{\mathcal{L}(X,Y)}$.
\end{proof}

The following corollary follows from Theorem \ref{thm:CP zero gap v2}.
\begin{corollary}
Let $X,Y$ be Hilbert spaces. Let $f:X\to\left(-\infty,+\infty\right],\ g:Y\to\left(-\infty,+\infty\right]$ and $L:X\to Y$ be a continuous linear operator. Assume that 
$a\geq 0,b\in \mathbb{R}$,
$c=d=0$ in the definitions of the classes of elementary functions $\Phi_{Q,a}$ and $\Psi_{Q,b}$ (\eqref{eq:class weak f} and \eqref{eq:class weak g}), respectively, and $\text{dom } g\cap L\left(\text{dom }f\right)\neq\emptyset$.
The following are equivalent.
\begin{enumerate}
\item For every $\varepsilon > 0$, there exist $x_{\varepsilon}\in X,\psi_\varepsilon \in\partial_{\varepsilon, \Psi_{Q,b}} g\left(Lx_{\varepsilon}\right), \phi_\varepsilon \in\partial_{\varepsilon, \Phi_{Q,a}}\left(f+b_{\varepsilon}\left\Vert L\cdot\right\Vert _{Y}^{2}\right)\left(x_{\varepsilon}\right)$ such that 
\begin{align*}
\begin{cases}
\phi_\varepsilon\left(x\right)+\psi_\varepsilon\left(Lx\right) \geq b \left\Vert Lx\right\Vert _{Y}^{2}-\varepsilon & \forall x\in X\\
\phi_\varepsilon\left(x_\varepsilon\right)+\psi_\varepsilon\left(Lx_\varepsilon \right) \leq  b \left\Vert Lx_\varepsilon\right\Vert _{Y}^{2}+\varepsilon,
\end{cases}
\end{align*}
where $\psi_\varepsilon (y) = b \Vert y\Vert^2_Y + \langle v_\varepsilon,y \rangle_Y$.
\item $\left(f+g\circ L\right)^{*}_{\Phi_{Q,a}} \left(0\right)={\displaystyle \inf_{\psi\in\Psi_{Q,b}}}\left(f+b\left\Vert L \cdot \right\Vert  _{Y}^{2}\right)^*_{\Phi_{Q,a}}  \left(-\langle L^* v,\cdot \rangle_X \right)+g^{*}_{\Psi_{Q,b}} \left(\psi\right)<+\infty$, where $\psi (y) = b \Vert y\Vert^2_Y + \langle v,y \rangle_Y$.
\end{enumerate}
\end{corollary}
\begin{proof}
%From the assumption 
Let $\varepsilon >0$. By (i),  
there exist $x_{\varepsilon}\in X,\psi_\varepsilon \in\partial_{\varepsilon, \Psi_{Q,b}} g\left(Lx_{\varepsilon}\right), \phi_\varepsilon \in\partial_{\varepsilon, \Phi_{Q,a}}\left(f+b_{\varepsilon}\left\Vert L\cdot\right\Vert _{Y}^{2}\right)\left(x_{\varepsilon}\right)$ such that
%definitions, if $\phi_{\varepsilon}\in\Phi_{Q,a}$, then $\phi_{\varepsilon}\left(x\right)=a \left\Vert x\right\Vert _{X}^{2}+\left\langle u_\varepsilon ,x\right\rangle _{X}$
%and $\psi_\varepsilon \in \Psi_{Q,b}$, means that $\psi_{\varepsilon}\left(y\right)=b \left\Vert y\right\Vert _{Y}^{2}+\left\langle v_\varepsilon ,y\right\rangle _{Y}$.
for all $x\in X$,
\begin{align*}
\phi_{\varepsilon}\left(x\right)+\psi_{\varepsilon}\left(Lx\right) = a \lVert x \rVert^2_X +\langle u_\varepsilon,x\rangle_X +b \lVert Lx \rVert^2_Y +\langle v_\varepsilon,Lx\rangle_Y \geq b \lVert Lx \rVert^2_Y-\varepsilon,
\end{align*}
where $u_\varepsilon\in X,v_\varepsilon \in Y$, i.e.,
\begin{equation}
\label{eq:cor 4.5}
a \lVert x \rVert^2_X +\langle u_\varepsilon + L^* v_\varepsilon,x\rangle_X +\varepsilon \geq 0,\quad \forall x\in X.
\end{equation}
Hence, it must be $u_\varepsilon=-L^{*}v_\varepsilon$. Thus, for $\phi_{\varepsilon}\in\partial_{\varepsilon, \Phi_{Q,a}} \left(f+b\left\Vert L \cdot \right\Vert _{Y}^{2}\right)\left(x_{\varepsilon}\right)$,
we can write 
\begin{align*}
f\left(x_{\varepsilon}\right)+b \left\Vert Lx_{\varepsilon}\right\Vert _{Y}^{2}+\left(f+b \left\Vert L\cdot\right\Vert _{Y}^{2}\right)^{*}_{\Phi_{Q,a}}\left(\phi_{\varepsilon}\right) & \leq\phi_{\varepsilon}\left(x_{\varepsilon}\right)+\varepsilon\\
\left(f+\psi\circ L\right)^{*}_{\Phi_{Q,a}} \left(0\right) & \leq-f\left(x_{\varepsilon}\right)-\psi\left(Lx_{\varepsilon}\right)+3\varepsilon,
\end{align*}
which is similar to the proof of Theorem \ref{thm:CP zero gap}, so the assertion follows from Theorem  \ref{thm:CP zero gap}. 
%so we follow the same
%argument and finish the proof.
\end{proof}
Observe that the above corollary does not hold for $\Phi_{Q,a}$ with $a<0$ due to formula \eqref{eq:cor 4.5}.

We give simple examples illustrating Theorem \ref{thm:CP zero gap v2}.

\begin{example}
\label{ex:zero duality gap 1}
\ 
\begin{itemize} 
\item Let $f\left(x\right)=\left(x+1\right)^{2},g\left(x\right)=4x^{2},L=Id$ and
\begin{align*}
\Phi & =\left\{ \phi\left(x\right)=-ax^{2}+bx,a\geq0,b\in\mathbb{R}\right\} ,\\
\Psi & =\left\{ \psi\left(x\right)=cx,c\in\mathbb{R}\right\},
\end{align*}
(c.f. Remark \ref{rmk:const}).

The conjugates of $f$ and $g$ are given as 
\begin{align*}
f^{*}_\Phi \left(\phi\right) & =\frac{\left(b-2\right)^{2}}{4\left(a+1\right)}-1,\quad a\geq0,b\in\mathbb{R}\\
g^{*}_\Psi\left(\psi\right) & =\frac{c^{2}}{16},\quad c\in\mathbb{R}.
\end{align*}

The $\varepsilon$-subdifferentials of $f$ and $g$ at $x_{0}$ are given as
\begin{align*}
\partial_{\varepsilon, \Phi}f\left(x_{0}\right) & =\left\{ \phi\left(x\right)=-ax^2 +bx:\left(a+1\right)\left(x_{0}-\frac{b-2}{2\left(a+1\right)}\right)^{2}\leq\varepsilon,a\geq0, b\in \mathbb{R}\right\}, \\
\partial_{\varepsilon, \Psi} g\left(x_{0}\right) & =\left\{ \psi\left(x\right) = cx:\left(2x_{0}-\frac{c}{4}\right)^{2}\leq\varepsilon, c\in \mathbb{R}\right\} .
\end{align*}

To verify condition (i) of Theorem \ref{thm:CP zero gap v2}, we find $x_\varepsilon\in \mathbb{R} ,\psi_\varepsilon \in\partial_{\varepsilon, \Psi} g\left(x_\varepsilon\right),\phi_\varepsilon \in\partial_{\varepsilon, \Phi} f\left(x_\varepsilon\right)$
such that 
\[
\begin{cases}
\phi_\varepsilon \left(x\right)+\psi_\varepsilon \left(x\right)\geq -\varepsilon & \forall x\in\mathbb{R}\\
\phi_\varepsilon \left(x_\varepsilon\right)+\psi_\varepsilon \left(x_\varepsilon\right)\leq \varepsilon
\end{cases}.
\]

Now consider the first inequality 
\[
\phi_\varepsilon \left(x\right)+\psi_\varepsilon \left(x\right)=-ax^{2}+\left(b+c\right)x\geq -\varepsilon\quad\forall x\in\mathbb{R},
\]
which gives $a=0,b+c=0$, so the second inequality i.e. $\phi_\varepsilon \left(x_\varepsilon\right)+\psi_\varepsilon\left(x_\varepsilon\right)\leq \varepsilon$, holds for
$x_\varepsilon\in\mathbb{R}$. Hence, Theorem \ref{thm:CP zero gap v2}-(i) holds.

Moreover, as $\psi_\varepsilon \in\partial_{\varepsilon, \Psi}g\left(x_\varepsilon\right),\phi_\varepsilon \in\partial_{\varepsilon, \Phi} f\left(x_\varepsilon\right)$,
we have 
\[
\begin{cases}
\left(x_\varepsilon-\frac{b-2}{2}\right)^{2}\leq\varepsilon\\
\left(2x_\varepsilon+\frac{b}{4}\right)^{2}\leq\varepsilon
\end{cases}.
\]

Let $\varepsilon\to0$, we obtain $b=2x_{0}+2,c=8x_{0}$. Since
$b+c=0$, we get $x_{0}=-\frac{1}{5}$. With $x_{0},a,b,c$ calculated above we get 
\begin{align*}
\inf_{x\in X} f(x) +g(x) = f\left(x_{0}\right)+g\left(x_{0}\right) & =\frac{4}{5},
\end{align*}
and 
\begin{align*}
\sup_{\psi\in \Psi} -(f+\psi)^*_\Phi (0) - g^*_\Psi (\psi) =-f^{*}_\Phi \left(\phi_\varepsilon\right)-g^{*}_\Psi\left(\psi_\varepsilon\right) & =\frac{80}{400}=\frac{4}{5}.
\end{align*}
Hence, we achieve zero duality gap. 
%In this example, condition \eqref{eq:CP cond 1} is always satisfied for all $\psi\in\Psi$ and $0$ whenever $\phi\in\Phi, \phi (x) = bx,b\in\mathbb{R}$ i.e.
%\[
%\phi\left(x\right)+\psi\left(x\right)=\left(b+c\right) x .
%\]

\item Now we reverse the roles of $\Phi$ and $\Psi$ i.e. 
\begin{align*}
\Phi & =\left\{ \phi\left(x\right)=cx,c\in\mathbb{R}\right\} ,\\
\Psi & =\left\{ \psi\left(x\right)=-ax^{2}+bx,a\geq0,b\in\mathbb{R}\right\} .
\end{align*}
 Then the conjugates of $f$ and $g$ are 
\begin{align*}
f^{*}_\Phi \left(\phi\right) & =\frac{\left(c-2\right)^{2}}{4}-1,\\
g^{*}_\Psi\left(\psi\right) & =\frac{b^{2}}{4\left(a+4\right)}.
\end{align*}

The $\varepsilon$-subdifferentials of $f$ and $g$ are 
\begin{align*}
\partial_{\varepsilon, \Phi} f\left(x_{\varepsilon}\right) & =\left\{ \phi\left(x\right)=cx:\left(x_{\varepsilon}-\frac{c-2}{2}\right)^{2}\leq\varepsilon,c\in\mathbb{R}\right\} ,\\
\partial_{\varepsilon, \Psi} g\left(x_{\varepsilon}\right) & =\left\{ \psi\left(x\right)=-ax^{2}+bx: \left(x_{\varepsilon}-\frac{b}{2\left(a+4\right)}\right)^{2}\leq \frac{\varepsilon}{a+4},a\geq0,b\in\mathbb{R}\right\} .
\end{align*}

We need to find $x_{\varepsilon}\in\mathbb{R}, \phi_\varepsilon \in\partial_{\varepsilon, \Phi}f\left(x_{\varepsilon}\right),\psi_\varepsilon \in\partial_{\varepsilon, \Psi} g\left(x_{\varepsilon}\right)$
such that 
\[
\begin{cases}
\phi_\varepsilon\left(x\right)+\psi_\varepsilon \left(x\right)\geq-\varepsilon & \forall x\in\mathbb{R}\\
\phi_\varepsilon\left(x_{\varepsilon}\right)+\psi_\varepsilon \left(x_{\varepsilon}\right)\leq\varepsilon
\end{cases}.
\]
The first inequality, for all $x\in \mathbb{R}$,
\[
\phi_\varepsilon\left(x\right)+\psi_\varepsilon \left(x\right)=-ax^{2}+\left(b+c\right)x\geq-\varepsilon,
\]
also implies $a=0,b+c=0$, so the second inequality is automatically
satisfied. Repeating the same steps as before, we achieve zero duality
gap at value $4/5$ which is the same in the previous case. However,
this time condition \eqref{eq:CP cond 1} does not hold for any $\psi\in\Psi$ and $0$,
because there is no $\phi\in\Phi$ such that
\[
\phi (x) = -\psi\left(x\right)=ax^{2}- bx,
\]
unless $a=0$.
\end{itemize}
\end{example}
In some cases, if Theorem \ref{thm:CP zero gap v2}-(i) is not satisfied, we need to check zero duality gap in a another way. Let us consider the following example.
\begin{example}
Let $X=Y=\mathbb{R}$, $f\left(x,y\right)=3x^{2}+2y^{2}$, $g\left(x\right)=-\left(x-1\right)^{2},L\left(x,y\right)=x-y$.
The classes of elementary functions are defined as
\begin{align*}
\Phi & =\left\{ \phi\left(x,y\right)=-a\left(x^{2}+y^{2}\right)+b_{1}x+b_{2}y,\ a\geq0,b_{1},b_{2}\in\mathbb{R}\right\} \\
\Psi & =\left\{ \psi\left(x\right)=-cx^{2}+dx,\ c\geq0,d\in\mathbb{R}\right\}, 
\end{align*}
(c.f. Remark \ref{rmk:const}).

We have
\begin{align*}
f^{*}_\Phi \left(\phi\right) & =\frac{b_{1}^{2}}{4\left(a+3\right)}+\frac{b_{2}^{2}}{4\left(a+2\right)},\\
\partial_{\varepsilon, \Phi} f\left(x_{\varepsilon},y_{\varepsilon}\right) & =\begin{cases}
\phi\left(x,y\right)=-a\left(x^{2}+y^{2}\right)+b_{1}x+b_{2}y,\\
\text{s.t. }\left(a+3\right)\left(x_{\varepsilon}-\frac{b_{1}}{2\left(a+3\right)}\right)^{2}+\left(a+2\right)\left(y_{\varepsilon}-\frac{b_{2}}{2\left(a+2\right)}\right)^{2}\leq\varepsilon
\end{cases},\\
g^{*}_\Psi\left(\psi\right) & =\begin{cases}
+\infty & 0\leq c<1\text{ or }c=1,d\neq2\\
1 & c=1,d=2\\
\frac{\left(d-2\right)^{2}}{4\left(c-1\right)} +1& c>1
\end{cases},\\
\partial_{\varepsilon, \Psi} g\left(x_{\varepsilon}\right) & =\begin{cases}
\psi\left(x\right)=-cx^{2}+dx & \left(c-1\right)\left(x_{\varepsilon}-\frac{d-2}{2\left(c-1\right)}\right)^{2}\leq\varepsilon,c>1\\
\psi\left(x\right)=-x^{2}+2x & c=1,d=2,\forall x_{\varepsilon}.
\end{cases}
\end{align*}

One option is to find, for $\varepsilon>0$, $x_\varepsilon,y_\varepsilon\in\mathbb{R},\phi_\varepsilon\in\partial_{\varepsilon, \Phi} f\left(x_\varepsilon,y_\varepsilon \right),\psi_\varepsilon\in\partial_{\varepsilon, \Psi} g\left(L\left(x_\varepsilon,y_\varepsilon\right)\right)$ such that 
\begin{equation}
\label{ex:2 eps system}
\begin{cases}
\phi_\varepsilon\left(x,y\right)+\psi_\varepsilon\left(L\left(x,y\right)\right)\geq-\varepsilon & \forall x,y\in\mathbb{R}\\
\phi_\varepsilon\left(x_{\varepsilon},y_{\varepsilon}\right)+\psi_\varepsilon \left(L\left(x_{\varepsilon},y_{\varepsilon}\right)\right)\leq\varepsilon.
\end{cases}
\end{equation}
 The first inequality is
\begin{align*}
\phi_\varepsilon \left(x,y\right)+\psi_\varepsilon \left(L\left(x,y\right)\right) & \geq-\varepsilon\\
-a\left(x^{2}+y^{2}\right)+b_{1}x+b_{2}y-c\left(x-y\right)^{2}+d\left(x-y\right)+\varepsilon & \geq0\\
-\left(a+c\right)\left(x^{2}+y^{2}\right)+2cxy+\left(b_{1}+d\right)x+\left(b_{2}-d\right)y+\varepsilon & \geq0,
\end{align*}
which gives us $a=c=0,b_{1}+d=b_{2}-d=0$, so the second inequality of \eqref{ex:2 eps system}
holds for any $x,y\in\mathbb{R}$. However, there are no element in $\partial_{\varepsilon, \Psi} g\left(L\left(x_{\varepsilon},y_{\varepsilon}\right)\right)$ such that $c=0$, and we cannot applied condition (i) of Corollary \ref{thm:CP zero gap v2} to determine zero duality gap.

One option is that instead of solving \eqref{ex:2 eps system}, we take $\psi_{\varepsilon}\in\partial_{\varepsilon, \Psi} g\left(L\left(x_{\varepsilon},y_{\varepsilon}\right)\right)$
and calculate $\partial_{\varepsilon, \Phi} \left(f+\psi_{\varepsilon}\circ L\right)\left(x_{\varepsilon},y_\varepsilon\right)$.
We consider the simplest case where $\psi_{\varepsilon}\left(x\right)=-x^{2}+2x$
and we calculate 
\[
\left(f+\psi_{\varepsilon}\circ L\right)^{*}_\Phi\left(\phi\right)=-\left(a+2\right)A^{2}-\left(a+1\right)B^{2}-2AB+\left(b_{1}-2\right)A+\left(b_{2}+2\right)B,
\]
where $A=\frac{ab_1-2a+b_1-b_2 -4}{2\left(a^2 +3a+1\right)},B=\frac{ab_2+2a-b_1+2b_2+6}{2\left(a^2+3a+1\right)}$,
for $a\geq0$. We examine whether $0\in\partial_{\varepsilon, \Phi}\left(f+\psi_{\varepsilon}\circ L\right)\left(x_{\varepsilon},y_{\varepsilon}\right)$
by checking the inequality 
\[
\left(f+\psi_{\varepsilon}\circ L\right)^{*}_\Phi\left(0\right)+\left(f+\psi_{\varepsilon}\circ L\right)\left(x_{\varepsilon},y_\varepsilon\right)\leq\varepsilon.
\]
 Substituting $\phi=0$ or $a=b_{1}=b_{2}=0$, we have 
\[
2x_{\varepsilon}^{2}+y_{\varepsilon}^{2}+2x_{\varepsilon}y_{\varepsilon}+2\left(x_{\varepsilon}-y_{\varepsilon}\right)+5\leq\varepsilon.
\]
This inequality has solution $x_{\varepsilon},y_{\varepsilon}$
as the left hand side is a parabola with minimum value $-\varepsilon$.
Thus, condition (i) of Theorem \ref{thm:CP zero gap} is satisfied, and we have zero duality gap for the conjugate duality.
\end{example}

Theorem \ref{thm:CP zero gap} is more general than Theorem \ref{thm:CP zero gap v2} but it requires more calculations than the latter. Sometimes, it is more convenient to calculate $\varepsilon$-subdifferential of $f$ and $g$ than $\varepsilon$-subdifferential of $(f+\psi_\varepsilon\circ L)$ where $\psi_\varepsilon\in\partial_{\varepsilon, \Psi} g(Lx_\varepsilon)$.

\section{Strong Duality for Conjugate Dual}
\label{sec:strong duality}

In this section, we investigate strong duality for the conjugate dual \eqref{prob:DCP}, i.e.,  we provide conditions for zero duality gap and the existence of solution to the dual problem \eqref{prob:DCP}; the respective conditions are expressed in terms of additivity of epigraphs of the conjugate functions.
\begin{definition}
The epigraph of a function $f:X\to (-\infty,+\infty]$ is a subset of $X\times\mathbb{R}$ defined as 
\begin{equation}
    \text{epi } f :=\left\{ (x,r) \in X\times \mathbb{R}: f(x) \leq r\right\}.
\end{equation}
\end{definition}
We have 
\begin{equation}
    \text{epi } f^*_\Phi :=\left\{ (\phi,r)\in\Phi\times\mathbb{R}: f^*_\Phi (\phi)\leq r\right\}.
\end{equation}
In consequence, if $(\phi,r)\in\text{epi } f^*_\Phi$, then $\phi\in\text{dom } f^*_\Phi$.
To investigate strong duality, together with \eqref{eq:CP cond 1},  we need the following condition: for a given $\phi\in\Phi$ and $\psi\in\Psi$
\begin{equation}
\label{eq:CP cond}
\phi + \psi\circ L \in \Phi.
\end{equation}

\begin{remark}
Note that \eqref{eq:CP cond} is satisfied when $\psi\circ L \in \text{lin } \Phi$ where $\text{lin } \Phi$ is the linear space generated by $\Phi$. 
\end{remark}

\begin{theorem}
\label{thm:epi-zero gap}
Let $X$ be a nonempty set and $Y$ be a vector space.
Let $f:X\to\left(-\infty,+\infty\right]$ and $g:Y\to\left(-\infty,+\infty\right]$
be proper functions. Let $\Phi$ and $\Psi$ be  sets of elementary functions defined on $X$ and $Y$, respectively.
%of $f$ and $g$.
Let $L:X\to Y$ be a mapping from $X$ to $Y$. Assume that $\text{dom }g\cap L\left(\text{dom }f\right)\neq\emptyset$ and consider the following conditions.
\begin{enumerate}
\item Every $(\phi,c_\phi)\in \text{epi } (f+g\circ L)^*_\Phi$, can be expressed as $(\phi,c_\phi)=(\varphi+\psi\circ L, c_\varphi+c_\psi)$, where $(\varphi,c_\varphi)\in \text{epi } f^*_\Phi$ and $(\psi,c_\psi)\in \text{epi }g^*_\Psi$.
\item For any $\phi\in\Phi$, it holds $\left(f+g\circ L\right)^{*}_\Phi\left(\phi\right)=\inf_{\psi\in\Psi}\left(f+\psi\circ L\right)^{*}_\Phi\left(\phi\right)+g^{*}_\Psi\left(\psi\right)$
and the infimum is attained.
\end{enumerate}
We have (i) $\Rightarrow$ (ii). Moreover, for any $\phi\in\Phi$, if $\bar{\psi}\in\Psi$ is a solution to the problem
\begin{equation}
\label{eq:CP general problem}
    \inf_{\psi\in\Psi}\left(f+\psi\circ L\right)^{*}_\Phi\left(\phi\right)+g^{*}_\Psi\left(\psi\right),
\end{equation}
and conditions \eqref{eq:CP cond}, \eqref{eq:CP cond 1} hold for $\bar{\psi}$ and $\phi$, then (ii) $\Rightarrow$ (i).
\end{theorem}
\begin{proof}
(i) $\Rightarrow$ (ii): 
If $\phi\notin \text{dom } (f+g\circ L)^*_\Phi$, then $(f+g\circ L)^*_\Phi (\phi) = +\infty$ and (ii) holds.
Consider the case $\phi\in\text{dom } (f+g\circ L)^*_\Phi$. By Theorem \ref{thm:CP weak dual}, for every $\phi\in\Phi$, we have
\begin{equation}
\label{eq:weak dual epi1}
\left(f+g\circ L\right)^{*}_\Phi\left(\phi\right)\leq \inf_{\psi\in\Psi}\left(f+\psi\circ L\right)^{*}_\Phi\left(\phi\right)+g^{*}_\Psi\left(\psi\right).
\end{equation}
We start by  proving the opposite inequality. Since
$\left(\phi,\left(f+g\circ L\right)^{*}_\Phi\left(\phi\right)\right)\in\text{epi }\left(f+g\circ L\right)^{*}_\Phi$,
by (i), there exist $\left(\varphi,c_{f}\right)\in\text{epi }f^{*}_\Phi $ and $\left(\psi_0,c_{g}\right)\in\text{epi } g^{*}_\Psi$ such that 
\begin{align*}
\phi\left(x\right) & =\varphi\left(x\right)+\psi_0\left(Lx\right),\quad\forall x\in X\\
\left(f+g\circ L\right)^{*}_\Phi\left(\phi\right) & =c_f+c_g.
\end{align*}
Thus
\begin{align*}
\left(f+g\circ L\right)^{*}_\Phi\left(\phi\right) & =c_f+c_g\\
 & \geq f^{*}_\Phi \left(\varphi\right)+g^{*}_\Psi\left(\psi_0\right)\\
 & =\sup_{x\in X}\left\{ \varphi\left(x\right)-f\left(x\right)\right\} +g^{*}_\Psi\left(\psi_0\right)\\
 & =\sup_{x\in X}\left\{ \phi\left(x\right)-\psi_0\left(Lx\right)-f\left(x\right)\right\} +g^{*}_\Psi\left(\psi_0\right)\\
 & =\left(f+\psi_0\circ L\right)^{*}_\Phi\left(\phi\right)+g^{*}_\Psi\left(\psi_0\right)\\
 & \geq\inf_{\psi\in\Psi}\left(f+\psi\circ L\right)^{*}_\Phi\left(\phi\right)+g^{*}_\Psi\left(\psi\right),
\end{align*}
and we have (ii). The infimum is attained from the above inequality, as 
\begin{align*}
+\infty > c_f +c_g =\left(f+g\circ L\right)^{*}_\Phi\left(\phi\right) &\geq \left(f+\psi_0\circ L\right)^{*}_\Phi\left(\phi\right)+g^{*}_\Psi\left(\psi_0\right)\\
& \geq\inf_{\psi\in\Psi}\left(f+\psi\circ L\right)^{*}_\Phi\left(\phi\right)+g^{*}_\Psi\left(\psi\right) \\
& \geq \left(f+g\circ L\right)^{*}_\Phi\left(\phi\right),
\end{align*}
so $\psi_0\in\Psi$ solves problem \eqref{eq:CP general problem}.

(ii) $\Rightarrow$ (i): From (ii), let $(\phi,c_f)\in\text{epi } f^*_\Phi$ and $\bar{\psi}$ be a solution to the problem \eqref{eq:CP general problem} for $\phi\in\Phi$. As $\left( f+\bar{\psi} \circ L\right)^*_\Phi (\phi) +g^*_\Psi (\bar{\psi}) <+\infty$, we can find $c_g\in\mathbb{R}$ such that $(\bar{\psi},c_g)\in \text{epi } g^*_\Psi$. Moreover, from \eqref{eq:CP cond}, $\phi + \bar{\psi} \circ L :=\varphi\in\Phi$, we have
\begin{align*}
c_f+c_g & \geq f^{*}_\Phi \left(\phi\right)+g^{*}_\Psi\left(\bar{\psi}\right)\\
 & =\sup_{x\in X} \left\{\varphi\left(x\right)-\bar{\psi} (Lx)-f\left(x\right) \right\}+g^{*}_\Psi\left(\bar{\psi}\right)\\
 & = (f+\bar{\psi}\circ L)^*_\Phi (\varphi) + g^*_\Psi (\bar{\psi})\\
 & \geq \left(f+g\circ L\right)^*_\Phi (\varphi).
\end{align*}
%where in the last inequality, we use Theorem \ref{thm:CP weak dual} for $\varphi\in\Phi$. 
This means $\left(\varphi,c_f+c_g\right) \in \text{epi } (f+g\circ L)^*_\Phi$. 

On the other hand, let $\left(\phi,r\right)\in\text{epi }\left(f+g\circ L\right)^{*}_\Phi$.
We want to prove that there exist $(\varphi,c_f)\in \text{epi } f^*_\Phi$ and $(\psi,c_g)\in \text{epi }g^*_\Psi$ such that 
\[
(\phi,r) = (\varphi+\psi\circ L, c_f + c_g).
\]
From (ii), for $\phi\in\Phi$, there exists $\bar{\psi}\in\Psi$ such that $\bar{\psi}$ is the solution to the problem \eqref{eq:CP general problem} i.e.
\[
\left(f+g\circ L\right)^{*}_\Phi\left(\phi\right) = \left(f+\bar{\psi}\circ L\right)^{*}_\Phi\left(\phi\right)+g^{*}_\Psi\left(\bar{\psi} \right).
\]
Therefore, we have
\begin{align*}
r & \geq\left(f+g\circ L\right)^{*}_\Phi\left(\phi\right)\\
r & \geq\left(f+\bar{\psi} \circ L\right)^{*}_\Phi\left(\phi\right)+g^{*}_\Psi\left(\bar{\psi} \right)\\
r-g^{*}_\Psi\left(\bar{\psi}\right) & \geq\sup_{x\in X} \left\{\phi\left(x\right)-\bar{\psi}\left(Lx\right)-f\left(x\right) \right\}.
\end{align*}
Condition \eqref{eq:CP cond 1} gives us $\varphi:=\phi-\bar{\psi}\circ L \in \Phi$. Then
\begin{align*}
r-g^{*}_\Psi\left(\bar{\psi}\right) & \geq\sup_{x\in X} \left\{\varphi\left(x\right)-f\left(x\right)\right\}= f^{*}_\Phi \left(\varphi\right).
\end{align*}
Therefore, $\left(\varphi,r-g^{*}_\Psi\left(\bar{\psi}\right)\right)\in\text{epi }f^{*}_\Phi$ and $\left(\bar{\psi},g^{*}_\Psi\left(\bar{\psi}\right)\right)\in\text{epi }g^{*}_\Psi$,
which means we can decompose $\left(\phi,r\right)$ into $\left(\varphi,r-g^{*}_\Psi\left(\bar{\psi}\right)\right)$ and $\left(\bar{\psi},g^{*}_\Psi\left(\bar{\psi}\right)\right)$. Hence, the proof is finished.
\end{proof}

\begin{remark}
\
\begin{itemize}
% \item Theorem \ref{thm:epi-zero gap} is different to \cite[Corollary 5.1]{Jey2007} because we have conditions \eqref{eq:CP cond}, \eqref{eq:CP cond 1}, using $\left(\psi\circ L, c_g\right)$ instead of using $\left(\psi, c_g\right)\in \text{epi } g^*_\Psi$. While in \cite{Bui2021,Jey2007}, 
%they define the 
%structure of $\Phi$ by using 
% weak star topology of $\Phi$ is used. We can also apply the topological approach, note
%define the same topology. 
% however, that the set 
%\[
%\left\{(\psi\circ L, c): (\psi,c)\in \text{epi } g^*_\Psi \right\}
%\]
%is not in $\Phi\times\mathbb{R}$, as $\psi\circ L \notin \Phi$.
\item Observe that condition (ii) of Theorem 5.2 implies the strong duality relationship between \eqref{prob:CP} and \eqref{prob:DCP}, i.e. zero duality gap holds 
\[-val (CP)=\left(f+g\circ L\right)^{*}_\Phi\left(0\right)=\inf_{\psi\in\Psi}\left(f+\psi\circ L\right)^{*}_\Phi\left(0\right)+g^{*}_\Psi \left(\psi\right) = - val (DCP),
\]
and the dual problem \eqref{prob:DCP} is solvable.
\item The assumption (i) in Theorem \ref{thm:epi-zero gap} does not coincide with the following
\[
\text{epi } (f+g\circ L)_\Phi^* = \text{epi } f_\Phi^* +\text{epi } (g\circ L)_\Phi^*,
\]
(see also \cite[Theorem 3.1]{Jey2007}), because we consider the set $B:=\left\{(\psi\circ L,r): (\psi,r)\in \text{epi } g_\Psi^*\right\}$ instead of $\text{epi } (g\circ L)_\Phi^*$, where $(\psi,r)\in\text{epi }g_\Psi^*$. Then the  assumption (i) of Theorem \ref{thm:epi-zero gap} is
\[
\text{epi } (f+g\circ L)_\Phi^* = \text{epi } f_\Phi^* +B,
\]
whenever $\phi +\psi\circ L\in\Phi$ holds true for any $\psi\in\text{dom }g^*_\Psi,\phi\in\text{dom }f^*_\Phi$. Examples of classes of elementary functions, where condition $\phi+\psi\circ L\in\Phi$ is satisfied, can be found in Example \ref{ex:3.2}-3.

\end{itemize}
\end{remark}

% We can further our result from Theorem \ref{thm:epi-zero gap}, but it requires a stronger conditions than \eqref{eq:CP cond}. 
\begin{corollary}
\label{cor:epi zero gap}
Let $X$ be a nonempty set and $Y$ be a vector space. Let $f:X\to\left(-\infty,+\infty\right],\ g:Y\to\left(-\infty,+\infty\right]$
and $L:X\to Y$ be a mapping from $X$ to $Y$. Let $\Phi,\Psi$
be sets of elementary functions on $X$ and $Y$, respectively. Assume that $\text{dom } g\cap L\left(\text{dom }f\right)\neq\emptyset$. Let
\begin{enumerate}
\item $\text{epi }\left(f+g\circ L\right)^{*}_\Phi=\text{epi }f^{*}_\Phi +\text{epi }\left(g\circ L\right)^{*}_\Phi$,
\item For $\phi\in \Phi$, $\inf_{\psi\in\Psi}\left(f+\psi\circ L\right)^{*}_\Phi\left(\phi\right)+g^{*}_\Psi\left(\psi\right)=\left(f+g\circ L\right)^{*}_\Phi\left(\phi\right)$ and the infimum is attained.
\end{enumerate}
Consider the following conditions.

\renewcommand{\labelenumi}{\alph{enumi}.}
\begin{enumerate}
\item For any pair $(\phi,r) \in \text{epi } (g\circ L)^*_\Phi$, there exists $\psi\in\Psi$ such that $\phi \leq \psi\circ L$ and $g^*_\Psi (\psi) \leq r$ 
\item There exists $\psi\in\Psi$ such that $g^*_\Psi (\psi) \leq 0$ and $g\circ L \leq \psi\circ L$.
\item For any pair $(\phi,r)\in \text{epi } \left( f+g\circ L\right)^*_\Phi$ and a given $\psi\in\Psi$, there exist $\phi_1,\phi_2\in\Phi$ such that $\phi-\psi\circ L \geq \phi_1$ and $\psi\circ L \geq \phi_2$
\item $\Phi-\Phi\subset\Phi$, for any $\varepsilon_{1},\varepsilon_{2}>0$ and a given $\psi\in \Psi$,
there exist $x_{\varepsilon_2}\in X,\phi_{\varepsilon_2}\in\partial_{\varepsilon_{2}, \Phi}\left(g\circ L\right)\left(x_{\varepsilon_2}\right)$
such that 
\[
\psi\left(Lx_{\varepsilon_2}\right)-\phi_{\varepsilon_2}\left(x_{\varepsilon_2}\right)+\varepsilon_{1}>\sup_{x\in X}\psi\left(Lx\right)-\phi_{\varepsilon_2}\left(x\right).
\]
\end{enumerate}
If $\left(a\right)$ or $\left(b\right)$ hold, then (i) $\Rightarrow$ (ii).
If $\left(c\right)$ or $\left(d\right)$ hold for $\bar{\psi}\in\Psi$ at which the infimum in (ii) is attained, then (ii) $\Rightarrow$ (i).
\end{corollary}

\begin{proof}
(i)$\Rightarrow$ (ii): For any
$\phi\in\Phi,\left(\phi,\left(f+g\circ L\right)^{*}_\Phi\left(\phi\right)\right)\in\text{epi }\left(f+g\circ L\right)^{*}_\Phi$.
There exist two pairs $\left(\phi_{1},c_{1}\right)\in\text{epi }f^{*}_\Phi$
and $\left(\phi_{2},c_{2}\right)\in\text{epi }\left(g\circ L\right)^{*}_\Phi$
such that $\phi=\phi_{1}+\phi_{2}$ and $c_{1}+c_{2}=\left(f+g\circ L\right)^*_\Phi \left(\phi\right)$.
Thanks to Theorem \ref{thm:CP weak dual}, for any $\phi\in\Phi$, we have 
\[
\inf_{\psi\in\Psi}\left(f+\psi\circ L\right)^{*}_\Phi\left(\phi\right)+g^{*}_\Psi\left(\psi\right)\geq\left(f+g\circ L\right)^{*}_\Phi\left(\phi\right),
\]
so we need to prove the reverse inequality. Consider 
\begin{equation}
\label{eq:cor CP-Phi only zero gap}
\inf_{\psi\in\Psi}\left(f+\psi\circ L\right)^{*}_\Phi\left(\phi\right)+g^{*}_\Psi\left(\psi\right)  \leq\left(f+\psi\circ L\right)^{*}_\Phi\left(\phi\right)+g^{*}_\Psi\left(\psi\right)
\end{equation}
Assume $\left(a\right)$. We can find $\psi_0\in\Psi$ such that $\phi_{2}\leq \psi_0\circ L$ and $g^*_\Psi (\psi_0) \leq c_2$. Note that 
\[
\phi_1 = \phi - \phi_2 \geq \phi - \psi_0\circ L.
\]
Hence,
\begin{align}
\inf_{\psi\in\Psi}\left(f+\psi\circ L\right)^{*}_\Phi\left(\phi\right)+g^{*}_\Psi\left(\psi\right) & \leq\left(f+\psi_0\circ L\right)^{*}_\Phi\left(\phi\right)+g^{*}_\Psi\left(\psi_0\right)\notag\\
 & \leq \sup_{x\in X}\left\{ \phi\left(x\right)-f\left(x\right)-\psi_0\left(Lx\right)\right\} +c_2\notag\\
 & \leq\sup_{x\in X}\left\{ \phi_{1}\left(x\right)-f\left(x\right)\right\} +c_2\notag\\
 & = f^{*}_\Phi \left(\phi_{1}\right)+c_2 \notag\\
 \label{eq:cor CP-Phi zero gap a}
 & \leq c_{1}+c_{2}=\left(f+g\circ L\right)^{*}_\Phi\left(\phi\right).
\end{align}
Inequality \eqref{eq:cor CP-Phi zero gap a} gives us
\begin{equation*}
(f+g\circ L)^*_\Phi (\phi) \leq\left(f+\psi_0\circ L\right)^{*}_\Phi\left(\phi\right)+g^{*}_\Psi\left(\psi_0\right) \leq (f+g\circ L)^*_\Phi (\phi) <+\infty,
\end{equation*}
and the infimum in (ii) is attained at $\psi_0$.

Now let $\left(b\right)$ hold. There exists $\psi_0\in\Psi$ such that $g\circ L \leq \psi_0\circ L$ and $g^*_\Psi (\psi_0)\leq 0$. We have
\begin{align}
\inf_{\psi\in\Psi}\left(f+\psi\circ L\right)^{*}_\Phi\left(\phi\right)+g^{*}_\Psi\left(\psi\right) & \leq\left(f+\psi_0\circ L\right)^{*}_\Phi\left(\phi\right)+g^{*}_\Psi\left(\psi_0\right)\notag\\
 & \leq\sup_{x\in X}\left\{ \phi_1\left(x\right)-f\left(x\right)\right\} +\sup_{x\in X}\left\{ \phi_2 (x)-\psi_0\left(Lx\right)\right\} \notag\\
 & \leq f^*_\Phi (\phi_1)+\sup_{x\in X}\left\{ \phi_2 (x)-g \left(Lx\right)\right\} \notag\\
 & \leq f^{*}_\Phi \left(\phi_{1}\right)+(g\circ L)^*_\Phi (\phi_2) \notag\\
 & \leq c_{1}+c_{2}=\left(f+g\circ L\right)^{*}_\Phi\left(\phi\right),
\end{align}
where we have used $\phi = \phi_1 +\phi_2$ in the second inequality. The attainment of the infimum of (ii) at $\psi_0\in\Psi$ follows in the same way as in the proof with condition (a).

(ii)$\Rightarrow $(i): Let $\left(\phi_{1},c_{1}\right)\in\text{epi }f^{*}_\Phi,\left(\phi_{2},c_{2}\right)\in\text{epi }\left(g\circ L\right)^{*}_\Phi$
we have 
\[
c_{1}+c_{2}\geq f^{*}_\Phi \left(\phi_{1}\right)+\left(g\circ L\right)^{*}_\Phi\left(\phi_{2}\right)\geq\left(f+g\circ L\right)^{*}_\Phi\left(\phi_{1}+\phi_{2}\right),
\]
so $\left(\phi_{1}+\phi_{2},c_{1}+c_{2}\right)\in\text{epi }\left(f+g\circ L\right)^{*}_\Phi$,
which means 
\[
\text{epi }\left(f+g\circ L\right)^{*}_\Phi\supset\text{epi }f^{*}_\Phi+\text{epi }\left(g\circ L\right)^{*}_\Phi.
\]
We want to prove the reverse inclusion. Take $\left(\phi,r\right)\in\text{epi }\left(f+g\circ L\right)^{*}_\Phi$.
From (ii), there exists $\bar{\psi}\in\Psi$ which is a solution to the problem,
\begin{align}
\inf_{\psi\in\Psi}\left(f+\psi\circ L\right)^{*}_\Phi\left(\phi\right)+g^{*}_\Psi\left(\psi\right) & =
\left(f+\bar{\psi}\circ L\right)^{*}_\Phi\left(\phi\right)+g^{*}_\Psi\left(\bar{\psi}\right)\notag\\
\label{eq:cor CP phi only zero gap 3}
& =\left(f+g\circ L\right)^{*}_\Phi\left(\phi\right)\leq r.
\end{align}

We assume $\left(c\right)$ holds. There exist $\phi_1,\phi_2\in\Phi$ such that $\phi\geq \phi_1 +\bar{\psi}\circ L$, $\bar{\psi}\circ L \geq \phi_2$, so that $(g\circ L)^*_\Phi (\phi_2) \leq g^*_\Psi (\bar{\psi})$ and
\begin{equation}
\label{eq:cor CP phi only zero gap 3-1}
\left(f+\bar{\psi}\circ L\right)^{*}_\Phi\left(\phi\right)=\sup_{x\in x}\phi\left(x\right)-\bar{\psi}\left(Lx\right)-f\left(x\right)\geq \sup_{x\in x}\phi_1 \left(x\right)-f\left(x\right)=f^*_\Phi (\phi_1).
\end{equation}
Combining this with \eqref{eq:cor CP phi only zero gap 3} and \eqref{eq:cor CP phi only zero gap 3-1}, we obtain
\[
f^{*}_\Phi \left(\phi_1\right)+(g\circ L)^{*}_\Phi \left(\phi_2\right) \leq \left(f+\bar{\psi}\circ L\right)^{*}_\Phi\left(\phi\right)+g^*_\Psi \left(\bar{\psi}\right)\leq r.
\]

By taking $\left(\phi_2,\left(g\circ L\right)^{*}_\Phi\left(\phi_2\right)\right)\in\text{epi }\left(g\circ L\right)^{*}_\Phi$,
we can make $\left(\phi_1,r-\left(g\circ L\right)^{*}_\Phi\left(\phi_2\right)\right)\in\text{epi }f^{*}_\Phi$.
This means 
\[
\left(\phi_2,\left(g\circ L\right)^{*}_\Phi\left(\phi_2\right)\right)+\left(\phi_1,r-\left(g\circ L\right)^{*}_\Phi\left(\phi_2\right)\right)=\left(\phi,r\right).
\]
Thus, we have
\begin{equation}
\label{eq:cor CP phi only epigraph}
\text{epi }\left(f+g\circ L\right)^{*}_\Phi\subset\text{epi }f^{*}_\Phi +\text{epi }\left(g\circ L\right)^{*}_\Phi.
\end{equation}

Now, let $\left(d\right)$ hold. Let $\bar{\psi}\in \Psi$ be a solution to the infimum problem in (ii). For every $\varepsilon_1,\varepsilon_2 >0$, there exist $x_{\varepsilon_2}\in X$, $\phi_{\varepsilon_2}\in\partial_{\varepsilon_{2}, \Phi} \left(g\circ L\right)\left(x_{\varepsilon_2}\right)$  such that
\[
-\sup_{x\in X}\left\{ \bar{\psi}\left(Lx\right) -\phi_{\varepsilon_2}\left(x\right)\right\} >-\bar{\psi}\left(Lx_{\varepsilon_2}\right)+\phi_{\varepsilon_2}\left(x_{\varepsilon_2}\right)-\varepsilon_{1}.
\]
We have
\begin{align}
\left(f+\bar{\psi}\circ L\right)^{*}_\Phi\left(\phi\right) & =\sup_{x\in X}\left\{ \phi\left(x\right)-f\left(x\right)-\bar{\psi}\left(Lx\right)\right\} \notag\\
\label{eq:cor CP phi only f+psi}
& \geq\sup_{x\in X}\left\{ \phi\left(x\right)-\phi_{\varepsilon_2}\left(x\right)-f\left(x\right)\right\} -\sup_{x\in X}\left\{ \bar{\psi}\left(Lx\right)-\phi_{\varepsilon_2}\left(x\right)\right\}
\end{align}

As $\phi-\phi_{\varepsilon_2} :=\varphi_{\varepsilon_2}\in\Phi$, we can write $\sup_{x\in X}\left\{ \phi\left(x\right)-\phi_{\varepsilon_2}\left(x\right)-f\left(x\right)\right\} =f^{*}_\Phi \left(\varphi_{\varepsilon_2}\right)$.
Since $g^{*}_\Psi\left(\bar{\psi}\right)\geq\bar{\psi} \left(y\right)-g\left(y\right)$ for any $y\in Y$, we can set $y=Lx_{\varepsilon_2}$. Combining this with \eqref{eq:cor CP phi only f+psi}
gives us
\begin{align*}
\left(f+\bar{\psi}\circ L\right)^{*}_\Phi\left(\phi\right)+g^{*}_\Psi\left(\bar{\psi}\right) & >f^{*}_\Phi \left(\varphi_{\varepsilon_2}\right)-\bar{\psi}\left(Lx_{\varepsilon_2}\right)+\phi_{\varepsilon_2}\left(x_{\varepsilon_2}\right)-\varepsilon_{1}+\bar{\psi}\left(Lx_{\varepsilon_2}\right)-g\left(Lx_{\varepsilon_2}\right)\\
 & >f^{*}_\Phi \left(\varphi_{\varepsilon_2}\right)+\phi_{\varepsilon_2}\left(x_{\varepsilon_2}\right)-g\left(Lx_{\varepsilon_2}\right)-\varepsilon_{1}\\
 & >f^{*}_\Phi \left(\varphi_{\varepsilon_2}\right)+(g\circ L)^*_\Phi (\phi_{\varepsilon_2})-\varepsilon_{1}-\varepsilon_2,
\end{align*}
as $\phi_{\varepsilon_2} \in \partial_{\varepsilon_2, \Phi} g\left(Lx_{\varepsilon_2}\right)$. From \eqref{eq:cor CP phi only zero gap 3}, we have
\begin{align*}
r & >\left(f+\bar{\psi}\circ L\right)^{*}_\Phi\left(\phi\right)+g^{*}_\Psi\left(\bar{\psi}\right)\\
 & >f^{*}_\Phi \left(\varphi_{\varepsilon_2}\right)+\left(g\circ L\right)^{*}_\Phi\left(\phi_{\varepsilon_2}\right)-\varepsilon_{2}-\varepsilon_{1}.
\end{align*}
We obtain $\left(\varphi_{\varepsilon_2},r+\varepsilon_1+\varepsilon_2-\left(g\circ L\right)^{*}_\Phi\left(\phi_{\varepsilon_2}\right)\right)\in\text{epi }f^{*}_\Phi$ and $\left( \phi_{\varepsilon_2}, \left(g\circ L\right)^*_\Phi (\phi_{\varepsilon_2}) \right)\in \text{epi } \left(g\circ L\right)^*_\Phi$.
This means $\text{epi }\left(f+g\circ L\right)^{*}_\Phi\subset\text{epi }f^{*}_\Phi+\text{epi }\left(g\circ L\right)^{*}_\Phi$, so we have \eqref{eq:cor CP phi only epigraph}.
\end{proof} 

\begin{remark}
\
\begin{itemize}
\item In Theorem \ref{thm:epi-zero gap} and Corollary \ref{cor:epi zero gap}, to obtain condition (ii), the assumptions mostly rely on $g$ and the set of elementary functions $\Psi$. Thus, with appropriate choices of $\Psi$, we can guarantee the attainment of the infimum in (ii). The motivation comes from the fact that in the construction of the conjugate dual, we perturb $g$ but not $f$.
\item If $0\in\Phi$, then the condition $\Phi - \Phi \subset \Phi$ implies symmetry of the set $\Phi$. Because we can take $0-\phi \in \Phi$ for any $\phi\in\Phi$.
\item It is true that $\text{epi }f^* =\text{supp }f$. To see this, let us take $(\phi,r)\in\text{epi } f^*$, then $\phi(x) - r \leq f(x)$ for all $x\in X$. As $\Phi$ is closed under addition of constant, we have $\phi - r\in\Phi$ and $\phi-r\in \text{supp } f$. Conversely, if $\phi\in\text{supp } f$ then $(\phi,0)\in \text{epi } f^*$ (see \cite{Jey2007}).
\item Corollary \ref{cor:epi zero gap}-(i) is closed to the additivity of the support of the conjugate in \cite[Theorem 5.1]{Jey2007}
\[
\text{supp}_\Phi (f+g\circ L) = \text{supp }_\Phi f +\text{supp }_\Phi (g\circ L),
\]
while in general, \cite[Proposition 2.2]{Jey2007}, we have
\[
\text{supp}_\Phi (f+g\circ L) = \text{co}_\Phi \left( \text{supp }_\Phi f +\text{supp }_\Phi (g\circ L) \right),
\]
where $\text{co}_\Phi A$ is the $\Phi$-convex hull of $A\subset \Phi$ i.e.
\[
\text{co}_\Phi A = \text{supp }_\Phi f_A, \quad \text{where } f_A (x) = \sup_{\phi\in A} \phi(x),\quad \forall x\in X.
\]
%\item In the case $L=Id$, the conditions for strong duality are given in terms of additivity of support which is equivalent to the epigraph in \cite[Corollary 5.2]{Jey2007}. When $L=Id,\Phi=\Psi$, Corollary \ref{cor:epi zero gap} coincides with \cite[Corollary 5.2]{Jey2007}. 
\item In the classical convex analysis, when $X$ and $Y$ are separated locally convex spaces. For lower semicontinuous proper convex functions $f,g$ and a continuous linear mapping $L$, \cite[Theorem 14.1]{Bot2009}, condition ensuring strong duality are expressed with the help of regularity condition $(RC)_5^\Phi$ as defined in \cite[Chapter 4, Section 14]{Bot2009} while in Theorem \ref{thm:epi-zero gap} and Corollary \ref{cor:epi zero gap}, we are using assumptions related to the additivity of epigraph of the conjugate. We refer the reader to Section 16 in \cite{Bot2009} for the discussion of relationship between different regularity conditions ensuring strong duality.
\end{itemize}
\end{remark}
We give an example of a convex problem satisfying the assumption (i) of Theorem \ref{thm:epi-zero gap}.
\begin{example}
Let $f\left(x,y\right)=x^{2}+y^{2},g\left(x\right)=2x^{2},L\left(x,y\right)=x+y$
and 
\begin{align*}
\Phi & =\left\{ \phi\left(x,y\right)=-a\left(x^{2}+y^{2}\right)+b_{1}x+b_{2}y,a\geq0,b_{1},b_{2}\in\mathbb{R}\right\} ,\\
\Psi & =\left\{ \psi\left(x\right)=cx,c\in\mathbb{R}\right\} .
\end{align*}
Condition \eqref{eq:CP cond}, $\phi+\psi\circ L\in\Phi$, always holds for all $\phi\in\Phi,\psi\in\Psi$ (see Example \ref{ex:3.2}-3). Let us denote 
\begin{align*}
\phi (x,y) &= -a (x^2 +y^2) +b_1 x +b_2 y, a\geq 0 \\ 
\varphi (x,y) & = -a_1 (x^2 +y^2) +b_{11} x +b_{12} y, a_1 \geq 0 \\
\psi (x) & = cx.
\end{align*}
We want to calculate 
\begin{align*}
\left(\phi,c_{1}\right)\in\text{epi }\left(f+g\circ L\right)_{\Phi}^{*} & \Leftrightarrow a\geq0,\frac{4b_{1}b_{2}-\left(3+a\right)\left(b_{1}^{2}+b_{2}^{2}\right)}{4\left(4-\left(3+a\right)^{2}\right)}\leq c_{1} \\
\left(\varphi,c_{2}\right)\in\text{epi }f_{\Phi}^{*} & \Leftrightarrow a_{1}\geq0,\frac{b_{11}^{2}+b_{12}^{2}}{4\left(a_{1}+1\right)}\leq c_{2}\\
\left(\psi,c_{3}\right)\in\text{epi }g_{\Psi}^{*} & \Leftrightarrow\psi\in\Psi, \frac{c^{2}}{8} \leq c_3
\end{align*}
Let us take $\left(\phi,c_{1}\right)\in\text{epi }\left(f+g\circ L\right)_{\Phi}^{*}$
then we want decompose $\phi=\varphi+\psi\circ L$ and $c_1=c_{2}+c_{3}$
where $\left(\varphi,c_{2}\right)\in\text{epi }f_{\Phi}^{*}$ and
$\left(\psi,c_{3}\right)\in\text{epi }g_{\Psi}^{*}$. We have the
following system 
\[
\begin{cases}
a=a_{1}\\
c+b_{11}=b_{1}\\
c+b_{12}=b_{2}
\end{cases}
\]
We can choose $c=c_{3}=0, a=a_1 =0, b_1 = b_2$ so $b_{1}=b_{11}=b_{2}=b_{12}$ and $c_1=c_{2}$. Taking $b_1 = 1 = c_1$ and assumption (i) of Theorem \ref{thm:epi-zero gap} is satisfied. Thus, we have strong duality for conjugate duality.
Note that if $\Phi$ is composed of affine functions only, we arrive at the same conclusion.
\end{example}
\section{Lagrange Dual}
\label{sec:Lagrange dual}

\subsection{Construction of Lagrangian Primal-Dual Problems}
Of equal importance is a problem to restate the results from \cite{Bed2020}, which connects the Conjugate Duality with Lagrange Duality. 
In the case we have an operator $L:X\to Y$ in the formulation of problem \eqref{prob:CP}. For other construction of Lagrangian dual proposed recently, see e.g. \cite{burachik2016asymptotic} and the references therein.

In this Section, we give new results for zero duality gap for Lagrange dual of composite problems.
To construct Lagrangian dual, we introduce the Lagrangian function $\mathcal{L}:X\times \Psi \to (-\infty,+\infty]$ as follow
\begin{equation}
\label{eq:Lagrangian}
\mathcal{L} (x,\psi) = f(x) + \psi(Lx) -g^*_\Psi (\psi),
\end{equation}
where $\Psi$ is a set of elementary functions defined on $Y$. In the case $\Psi$ is a convex set then $\mathcal{L}(x,\psi)$ is concave with respect to $\psi\in\Psi$.
The Lagrangian dual is 
\begin{equation}
\label{prob:LD}
val (LD) = \sup_{\psi\in\Psi}\inf_{x\in X} \mathcal{L} (x,\psi) = \sup_{\psi\in\Psi}\inf_{x\in X} f(x) + \psi(Lx) -g^*_\Psi (\psi), \tag{LD}
\end{equation}
and the corresponding Lagrangian primal 
\begin{equation}
\label{prob:LP}
val (LP) :=\inf_{x\in X}\sup_{\psi\in\Psi} \mathcal{L} (x,\psi) = \inf_{x\in X}\sup_{\psi\in\Psi} f(x) + \psi(Lx) -g^*_\Psi (\psi), \tag{LP}
\end{equation}
Let us state the weak duality for Lagrangian duality. 
\begin{theorem}
\label{thm:Lagrange weak dual}
Let $f:X\to (-\infty,+\infty],\ g:Y\to (-\infty,+\infty]$ and $L:X\to Y$ be a mapping. Let $\Phi,\Psi$ be sets of elementary functions defined on $X$ and $Y$, respectively. It holds
\[
\sup_{\psi\in\Psi}\inf_{x\in X} \mathcal{L} (x,\psi) \leq \inf_{x\in X}\sup_{\psi\in\Psi} \mathcal{L} (x,\psi)
\]
\end{theorem}
This result holds true for any functions not necessarily of the form \eqref{eq:Lagrangian}.
First we notice that Lagrangian primal \eqref{prob:LP} is not equivalent to the composite problem \eqref{prob:CP} as
\begin{align*}
\inf_{x\in X}\sup_{\psi\in\Psi} f(x) + \psi(Lx) -g^*_\Psi (\psi) & = \inf_{x\in X} f(x) + \sup_{\psi\in\Psi} \psi(Lx) - g^*_\Psi (\psi) \\
& = \inf_{x\in X} f(x) + g^{**}_\Psi (Lx) \\
& \leq \inf_{x\in X} f(x) + g (Lx).
\end{align*}
Even though the Lagrange dual and conjugate dual are the same, the primal problems are different. Since our main focus is \eqref{prob:CP}, we discuss the assumptions which make these two primal problems equivalent. Clearly, we can assume 
\begin{equation}
\label{eq:CP-LP cond 1}
\inf_{x\in X} f(x) + g^{**}_\Psi (Lx) = \inf_{x\in X} f(x) + g (Lx).
\end{equation}
Usually, in the classical convex approach, $g$ is lsc and convex iff $g^{**}_\Psi = g$ so condition \eqref{eq:CP-LP cond 1} holds. Conversely, we have the following.
\begin{proposition}
Let $f:X\to (-\infty,+\infty],\ g:Y\to (-\infty,+\infty]$ where $X$ is a nonempty set and $Y$ is a vector space. Let $L:X\to Y$ be a mapping. Let $\Phi, \Psi$ be the sets of elementary functions  defined on $X$ and $Y$, respectively. Assume \eqref{eq:CP-LP cond 1} holds, if there exists an $x_0\in X$ such that 
\[
\inf_{x\in X} f(x) +g(Lx) \geq f(x_0) +g(Lx_0),
\]
then $g$ is $\Psi$-convex at $L x_0$.
\end{proposition}
\begin{proof}
We have
\[
\inf_{x\in X}f\left(x\right)+g\left(Lx\right) \geq f\left(x_{0}\right)+g\left(Lx_{0}\right).
\]
From \eqref{eq:CP-LP cond 1} we have 
\[
f\left(x_{0}\right)+g\left(Lx_{0}\right) \leq \inf_{x\in X}f\left(x\right)+g^{**}_\Psi \left(Lx\right)\leq f\left(x_{0}\right)+g^{**}_\Psi \left(Lx_{0}\right),
\]
which implies $g\left(Lx_{0}\right)\leq g^{**}_\Psi \left(Lx_{0}\right)$.
From the definition of biconjugate function, $g^{**}_\Psi\leq g$, so we
have $g\left(Lx_{0}\right)=g^{**}_\Psi \left(Lx_{0}\right)$. Thus, $g$
is $\Psi$-convex at $L x_{0}$. 
\end{proof}
%Note that if we assume $\inf_{x\in X} f(x) + g^{**}_\Psi (Lx)$ is attained, we cannot get the  conclusion of the above proposition.

\subsection{Lagrange Zero Duality Gap}
In the present subsection, we discuss zero duality gap for Lagrange duality. We follow the result of \cite{Bed2020} and exploit the intersection property to prove zero duality gap. In our context, Theorem 6.1 in \cite{Bed2020} takes the following form.
\begin{theorem}
\label{thm:intersection-1}
(Theorem 6.1 \cite{Bed2020})
Let $X = Y$ be a vector space, $\Phi$ be a convex set of elementary functions defined on $X$. Let $\mathcal{L} (x,\psi)$ be Lagrangian given by \eqref{eq:Lagrangian} with $L=Id$. Then the following are equivalent.
\begin{enumerate}
\item For every $\alpha<\inf_{x\in X}\sup_{\psi\in\Phi}\mathcal{L}\left(x,\psi\right)$,
there exists $\psi_{1},\psi_{2}\in\Phi$ and $\phi_{1}\in\text{supp }\mathcal{L}\left(\cdot,\psi_{1}\right),\phi_{2}\in\text{supp }\mathcal{L}\left(\cdot,\psi_{2}\right)$
such that $\phi_{1},\phi_{2}$ have the intersection property at level $\alpha$; i.e., for all $t\in\left[0,1\right]$ 
\[
\left[t\phi_{1}+\left(1-t\right)\phi_{2}<\alpha\right]\cap\left[\phi_{1}<\alpha\right]=\emptyset\text{ or }\left[t\phi_{1}+\left(1-t\right)\phi_{2}<\alpha\right]\cap\left[\phi_{2}<\alpha\right]=\emptyset,
\]
where $\left[\phi_{1}<\alpha\right]:=\left\{ x\in X:\phi_{1}\left(x\right)<\alpha\right\} $.

\item $val(LP) = val (DCP)$ i.e. $\inf_{x\in X}\sup_{\psi\in\Phi}\mathcal{L}\left(x,\psi\right)=\sup_{\psi\in\Phi}\inf_{x\in X}\mathcal{L}\left(x,\psi\right).$
\end{enumerate}
\end{theorem}
The proof can be found in \cite{Syga2018}. 

%\begin{remark}
%\begin{itemize}
%\
%\item In the conjugate dual, from assumption (1) of Theorem \ref{thm:CP zero gap}, to get zero duality gap for any $\varepsilon>0$, we have to find $x_\varepsilon\in X$ and $\psi_\varepsilon\in\partial_{\varepsilon, \Psi} g(Lx_\varepsilon)$ such that 
%\[
%(f+\psi_\varepsilon \circ L)^*_\Phi (0) \leq (f+\psi_\varepsilon\circ L)(x_\varepsilon) +\varepsilon.
%\]
%By Fenchel-Young inequality, this means that $0\in\partial_{\varepsilon, \Phi} (f+\psi_\varepsilon \circ L)(Lx_\varepsilon)$. Thus there is a connection between $\Psi$ and $\Phi$. Another example of connection between $\Psi$ and $\Phi$ are conditions \eqref{eq:CP cond} and \eqref{eq:CP cond 1}. 

%In Lagrangian duality, the connection between $\Psi$ and $\Phi$ lies in the intersection property. We need to find four functions to satisfy Theorem \ref{thm:intersection-1}-(i) for every $\alpha\in\mathbb{R}$. 
%This is less restrictive compare to \eqref{eq:CP cond} and \eqref{eq:CP cond 1} where we have to assume for the whole class $\Psi$ and $\Phi$. 
%Moreover, we also need to assume \eqref{eq:CP-LP cond 1} for \eqref{prob:CP} to become Lagrangian primal.
%\item 
\begin{remark}
In the case where for every $\alpha <\inf_{x\in X} \sup_{\psi\in\Psi}\mathcal{L} (x,\psi)$, we have $\psi_1 =\psi_2$, i.e. $\phi_{1},\phi_{2}$ belong to the same support set $\text{supp }\mathcal{L }\left(\cdot,\psi_{1}\right)$ for some $\psi_1 \in \Psi$, the assumption of concavity of the Lagrangian can be removed. (cf. \cite{Syga2016}).
\end{remark}
%\end{itemize}
%\end{remark}
For convenience, we state a result from \cite{Bed2020}.
\begin{lemma}
\label{lm:intersection prop}
(Lemma 6.2 \cite{Bed2020})
Let $X$ be a set, $\alpha\in\mathbb{R}$ and let
$\phi_{1},\phi_{2}:X\to\mathbb{R}$. Two functions $\phi_{1},\phi_{2}$
have the intersection property at level $\alpha$ if and only if there exists $t_{0}\in\left[0,1\right]$ such that $t_{0}\phi_{1}\left(x\right)+\left(1-t_{0}\right)\phi_{2}\left(x\right)\geq\alpha$
for all $x\in X$. 
\end{lemma}
Observe that in Theorem 6.3, the class of elementary functions $\Phi$ is arbitrary. Below, we relate the intersection property to the lower semi-continuity of the optimal value function, $V$, in the case where elementary functions satisfy additional conditions. For other results along this line, see e.g. \cite{rubinov2002zero,huang2005further}.

Let us define the optimal value function $V:Y\to [-\infty, +\infty)$ as follow
\begin{equation}
    V(y) := \inf_{x\in X} \beta (x,y)
\end{equation}
where the function $\beta$ is defined at the beginning of Section \ref{sec:Conjugate dual}. Similar to \eqref{eq:CP-LP cond 1}, we also define the following condition for any $y\in Y$
\begin{equation}
\label{eq:CP-LP cond 2}
\inf_{x\in X} f(x) + g^{**}_\Psi (Lx+y) = \inf_{x\in X} f(x) + g (Lx+y).
\end{equation}
A function $V:Y \to [-\infty,+\infty)$ is lower semi-continuous at a point $y_0\in Y$, if for every $\varepsilon>0$, there exists a neighborhood $W(y_0)$ such that 
\begin{equation}
    V(y) > V(y_0) -\varepsilon, \quad \forall y\in W(y_0).
\end{equation}

The following theorem below connects Theorem \ref{thm:CP zero gap}, the intersection property and the lower semi-continuity of the objective function at $y_0 = 0$.

\begin{theorem}
\label{thm:Lagrange intersection}
Let $f:X\to\left(-\infty,+\infty\right],\ g:Y\to\left(-\infty,+\infty\right]$.
Let $L:X\to Y$ be a mapping from $X$ to $Y$ with $\text{dom } g\cap L\left(\text{dom }f\right)\neq\emptyset$. Let $\Phi, \Psi$ be sets of elementary functions defined on $X$ and $Y$, respectively. Let $\mathcal{L} (x,\psi)$ be the Lagrangian defined in \eqref{eq:Lagrangian}. Assume $\Psi$ is
convex, $0\in\Phi$ and $\inf_{x\in X} f(x) +g(Lx) <+\infty$. Consider the following.
\begin{enumerate}
\item For every $\varepsilon>0$, there exist $x_\varepsilon\in X$ and $\psi_\varepsilon\in\partial_{\varepsilon, \Psi} g\left(Lx_\varepsilon\right)$
such that $0\in\partial_{\varepsilon, \Phi}\left(f+\psi_\varepsilon\circ L\right)\left(x_\varepsilon\right)$.
\item For every $\alpha<\inf_{x\in X}\sup_{\psi\in\Psi}\mathcal{L}\left(x,\psi\right)$,
there exists $\phi_{1},\phi_{2}\in\Phi$, $\psi_{1},\psi_{2}\in\Psi$,
$\phi_{1}\in\text{supp}\mathcal{L}\left(,\psi_{1}\right),\phi_{2}\in\text{supp}\mathcal{L}\left(,\psi_{2}\right)$
such that the intersection property holds for $\phi_{1},\phi_{2}$
at level $\alpha$.

\item The function $V$ is lower semi-continuous at 0.
\end{enumerate}
We have (i)$\Rightarrow$(ii). If condition \eqref{eq:CP-LP cond 1} holds, then (ii)$\Leftrightarrow$(i).
Moreover, if \eqref{eq:CP-LP cond 1} holds and $\Psi$ is composed of lower semi-continuous functions, then (ii)$\Rightarrow$(iii). (iii)$\Rightarrow$(ii) when $\Psi$ is a collection of continuous functions and \eqref{eq:CP-LP cond 2} holds.
\end{theorem}

\begin{proof}
(i) $\Rightarrow$ (ii): For every $\alpha\in\mathbb{R}$, let $\alpha<\inf_{x\in X}\sup_{\psi\in\Psi}\mathcal{L}\left(x,\psi\right)$. 
For every $\varepsilon>0$,
we can find $x_\varepsilon\in X$ and $\psi_\varepsilon\in\partial_{\varepsilon, \Psi}g\left(Lx_\varepsilon\right)$
such that $0\in\partial_{\varepsilon, \Phi}\left(f+\psi_\varepsilon\circ L\right)\left(x_\varepsilon\right)$,
\begin{equation}
\label{eq:Lagrange intersection 1}
\left(\forall z\in X\right)\quad f\left(z\right)+\psi_\varepsilon\left(Lz\right)\geq f\left(x_\varepsilon\right)+\psi_\varepsilon\left(Lx_\varepsilon\right)-\varepsilon.
\end{equation}
We also have  $\psi_\varepsilon\in\partial_{\varepsilon, \Psi}g\left(Lx_\varepsilon\right)$ so $g(Lx_\varepsilon) + g^*_\Psi (\psi_\varepsilon) \leq \psi_\varepsilon(Lx_\varepsilon) + \varepsilon$. 
Putting this into \eqref{eq:Lagrange intersection 1}, 
\begin{align*}
f\left(z\right)+\psi_\varepsilon\left(Lz\right) & \geq f\left(x_\varepsilon\right)+\psi_\varepsilon\left(Lx_\varepsilon\right)-\varepsilon\\
 & \geq f\left(x_\varepsilon\right)+g\left(Lx_\varepsilon\right)+g^{*}_\Psi\left(\psi_\varepsilon\right)-2\varepsilon,
\end{align*}
so 
\begin{equation}
\label{eq:thm 3.9}
\mathcal{L}\left(z,\psi_\varepsilon\right)=f\left(z\right)+\psi_\varepsilon\left(Lz\right)-g^{*}_\Psi\left(\psi_\varepsilon\right)\geq f\left(x_\varepsilon\right)+g\left(Lx_\varepsilon\right)-2\varepsilon.
\end{equation}
As $\varepsilon>0$ is arbitrary, we choose $\varepsilon = (f(x_\varepsilon) + g(Lx_\varepsilon) - \alpha)/2 >0$ so that $\mathcal{L} (z,\psi_\varepsilon) > \alpha$.
Pick $\phi_{1}\in\Phi$ and $\psi_{1}\in\Psi$ such that $\phi_{1}\in\text{supp }\mathcal{L}\left(\cdot,\psi_{1}\right)$.
By choosing $\phi_{2}=\alpha\in\Phi$, we have $\phi_{2}\in\text{supp }\mathcal{L}\left(\cdot,\psi_\varepsilon\right)$.
From Lemma \ref{lm:intersection prop} with $t=1$ we have 
\[
\mathcal{L}\left(z,t\psi_\varepsilon+\left(1-t\right)\psi_{1}\right)=\mathcal{L}\left(z,\psi_\varepsilon\right)>t\alpha+\left(1-t\right)\phi_{1}\left(x\right)=\alpha,
\]
so the intersection property holds for $\phi_{1},\phi_{2}$ at level
$\alpha$. As $\alpha$ is arbitrary, we have (ii). 

(ii) $\Rightarrow$ (i): As \eqref{eq:CP-LP cond 1} holds, we have $val(CP)= val(LP)$.
From the primal problem
\eqref{prob:CP}, for every $\varepsilon>0$, there exists an $x_{\varepsilon}\in X$
such that 
\[
\inf_{x\in X}f\left(x\right)+g\left(Lx\right)>f\left(x_{\varepsilon}\right)+g\left(Lx_{\varepsilon}\right)-\varepsilon.
\]
Denoting $\alpha:= val (\text{CP}) -\varepsilon$, we have $\alpha<val\left(\text{CP}\right)$. There exist $\phi_{1},\phi_{2}\in\Phi$ and $\psi_1,\psi_2\in \Psi$ such that $\phi_1 \in \text{supp } \mathcal{L} (\cdot,\psi_1),\phi_2 \in \text{supp } \mathcal{L} (\cdot,\psi_2)$ and the intersection property holds at level $\alpha$. Lemma \ref{lm:intersection prop} gives us a $t_{0}\in\left[0,1\right]$ such that

\begin{equation}
\label{eq:Lagrange intersection 2}
t_{0}\phi_{1}\left(x\right)+\left(1-t_{0}\right)\phi_{2}\left(x\right)\geq\alpha\quad\forall x\in X.
\end{equation}
Because $\phi_{1}\in\text{supp}\mathcal{L}\left(\cdot,\psi_{1}\right),\phi_{2}\in\text{supp}\mathcal{L}\left(\cdot,\psi_{2}\right)$, we have
\[
\mathcal{L}\left(x,\psi_{1}\right)\geq\phi_{1}\left(x\right)\text{ and }\mathcal{L}\left(x,\psi_{2}\right)\geq\phi_{2}\left(x\right)\quad\forall x\in X.
\]
As $\Psi$ is a convex set, $t_{0}\psi_{1}+\left(1-t_{0}\right)\psi_{2}\in\Psi$. Moreover, $\mathcal{L}\left(x,\psi\right)$
is concave in the second variable, i.e. for all $t\in\left[0,1\right]$
and $\psi_{1},\psi_{2}\in\Psi$
\begin{align}
\mathcal{L}\left(x,t\psi_{1}+\left(1-t\right)\psi_{2}\right) & =f\left(x\right)+t\psi_{1}\left(L x\right)+\left(1-t\right)\psi_{2}\left(Lx\right)-g^{*}_\Psi\left(t\psi_{1}+\left(1-t\right)\psi_{2}\right)\notag\\
& \geq t\left[f\left(x\right)+\psi_{1}\left(Lx \right)-g^{*}_\Psi\left(\psi_{1}\right)\right]+\left(1-t\right)\left[f\left(x\right)+\psi_{2}\left(Lx \right)-g^{*}_\Psi\left(\psi_{2}\right)\right]\notag\\
 \label{eq:Lagrange intersection 3}
 & =t\mathcal{L}\left(x,\psi_{1}\right)+\left(1-t\right)\mathcal{L}\left(x,\psi_{2}\right),
\end{align}
where the inequality holds due to the concavity of  $-g^{*}_\Psi$. Combining inequality \eqref{eq:Lagrange intersection 3} with \eqref{eq:Lagrange intersection 2}, we get 
\begin{align*}
\mathcal{L}\left(x,t_{0}\psi_{1}+\left(1-t_{0}\right)\psi_{2}\right) & \geq t_{0}\mathcal{L}\left(x,\psi_{1}\right)+\left(1-t_{0}\right)\mathcal{L}\left(x,\psi_{1}\right) \\
& \geq t_{0}\phi_{1}\left(x\right)+\left(1-t_{0}\right)\phi_{2}\left(x\right)\geq\alpha \\
& =val (\text{CP}) -\varepsilon\\
& =\inf_{x\in X}f\left(x\right)+g\left(Lx\right)-\varepsilon\\
& >f\left(x_{\varepsilon}\right)+g\left(Lx_{\varepsilon}\right)-2\varepsilon.
\end{align*}
% From the assumption $\alpha<val\left(\text{CP}\right)<+\infty$,
% we can set $A=val\left(\text{CP}\right)$. 
We obtain 
\[
\mathcal{L}\left(x,t_{0}\psi_{1}+\left(1-t_{0}\right)\psi_{2}\right)>f\left(x_{\varepsilon}\right)+g\left(Lx_{\varepsilon}\right)-2\varepsilon\quad\forall x\in X
\]
By denoting $\bar{\psi}=t_{0}\psi_{1}+\left(1-t_{0}\right)\psi_{2}$, we get

\begin{equation}
\label{eq:thm 3.9 2}
\mathcal{L}\left(x,\bar{\psi}\right)=f\left(x\right)+\bar{\psi}\left(Lx\right)-g^{*}_\Psi\left(\bar{\psi}\right)>f\left(x_{\varepsilon}\right)+g\left(Lx_{\varepsilon}\right)-2\varepsilon.
\end{equation}
Rearranging both sides, we obtain 
\[
f\left(x\right)-f\left(x_{\varepsilon}\right)>-\bar{\psi}\left(Lx\right)+g\left(Lx_{\varepsilon}\right)+g^{*}_\Psi\left(\bar{\psi}\right)-2\varepsilon.
\]
We have $g(Lx) + g^*_\Psi (\psi) \geq \psi(Lx)$ for all $x\in X$ and $\psi\in \Psi$. From the previous inequality, 
\[
f\left(x\right)-f\left(x_{\varepsilon}\right)>-\bar{\psi}\left(Lx\right)+\bar{\psi}\left(Lx_{\varepsilon}\right)-2\varepsilon.
\]
Because, in general, $-\bar{\psi}\circ L$ does not belong to $\Phi$, we cannot
claim that $-\bar{\psi}\circ L\in\partial_{2\varepsilon, \Phi} f\left(x_{\varepsilon}\right)$,
but we have 
\[
f\left(x\right)+\bar{\psi}\left(Lx\right)-f\left(x_{\varepsilon}\right)-\bar{\psi}\left(Lx_{\varepsilon}\right)>-2\varepsilon,
\]
which means $0\in\partial_{2\varepsilon, \Phi}\left(f+\bar{\psi}\circ L\right)\left(x_{\varepsilon}\right)$.
On the other hand, from \eqref{eq:thm 3.9 2}, for any $x\in X$, we have
\[
f\left(x\right)+\bar{\psi}\left(Lx\right)-g^{*}_\Psi\left(\bar{\psi}\right)>f\left(x_{\varepsilon}\right)+g\left(Lx_{\varepsilon}\right)-2\varepsilon.
\]
By choosing $x=x_{\varepsilon}$, we can write 
\[
\bar{\psi}\left(Lx_{\varepsilon}\right)-g^{*}_\Psi\left(\bar{\psi}\right)>g\left(Lx_{\varepsilon}\right)-2\varepsilon,
\]
or 
\[
\bar{\psi}\left(Lx_{\varepsilon}\right)+2\varepsilon>g\left(Lx_{\varepsilon}\right)+g^{*}_\Psi\left(\bar{\psi}\right).
\]
This gives $\bar{\psi}\in\partial_{2\varepsilon, \Psi}g\left(Lx_{\varepsilon}\right)$.
Thus, (i) holds.

(ii)$\Rightarrow$ (iii):
As the intersection property holds for every $\alpha\in\mathbb{R}$, for every $\varepsilon>0$, we can take
$\alpha_{\varepsilon}=\inf_{x\in X}\sup_{\psi\in\Psi}\mathcal{L}\left(x,\psi\right)-\varepsilon/2 = V(0) -\varepsilon/2$.
Now we can find $\phi_{1},\phi_{2}\in\Phi,\psi_{1},\psi_{2}\in\Psi$
such that $\phi_{1}\left(\cdot\right)\leq\mathcal{L}\left(\cdot,\psi_{1}\right),\phi_{2}\leq\mathcal{L}\left(\cdot,\psi_{2}\right)$
for all $x\in X$, and $\phi_{1},\phi_{2}$ satisfy the intersection
property at level $\alpha_{\varepsilon}$. Using Lemma \ref{lm:intersection prop}, there
exists $t_{0}\in\left[0,1\right]$ such that for all $x\in X$,
\begin{align}
\alpha_\varepsilon & \leq t_{0}\phi_{1}\left(x\right)+\left(1-t_{0}\right)\phi_{2}\left(x\right)\leq t_{0}\mathcal{L}\left(x,\psi_{1}\right)+\left(1-t_{0}\right)\mathcal{L}\left(x,\psi_{2}\right)\nonumber \\
 & \leq\mathcal{L}\left(x,t_{0}\psi_{1}+\left(1-t_{0}\right)\psi_{2}\right).\label{eq:IP lsc 1-2}
\end{align}

On the other hand, we have 
\begin{align*}
\mathcal{L}\left(x,\psi\right) & =f\left(x\right)+\psi\left(Lx\right)-g_{\Psi}^{*}\left(\psi\right)\\
 & =\inf_{y\in Y}f\left(x\right)+g\left(y\right)+\psi\left(Lx\right)-\psi\left(y\right)\\
 & \leq f\left(x\right)+g\left(Lx+y\right)+\psi\left(Lx\right)-\psi\left(Lx+y\right),
\end{align*}
where $x\in X,y\in Y$. Denoting $\psi_{0}=t_{0}\psi_{1}+\left(1-t_{0}\right)\psi_{2}\in\Psi$,
as $\Psi$ is a convex set, (\ref{eq:IP lsc 1-2}) gives us 
\begin{align*}
\alpha_\varepsilon & \leq f\left(x\right)+g\left(Lx+y\right)+\psi_{0}\left(Lx\right)-\psi_{0}\left(Lx+y\right)\\
\alpha_\varepsilon +\psi_{0}\left(Lx+y\right)-\psi_{0}\left(Lx\right) & \leq f\left(x\right)+g\left(Lx+y\right),
\end{align*}
for any $x\in X,y\in Y$. Since $\psi_{0}\in\Psi$ is lower semi-continuous,
the function $\psi_{x}\left(y\right):=\psi_{0}\left(Lx+y\right)$
is also lower semi-continuous for all $y\in Y$. There exists a neighborhood
$W\left(0\right)\subset Y$ such that 
\[
\psi_{x}\left(y\right)\geq\psi_{x}\left(0\right)-\varepsilon/2,\quad\forall y\in W\left(0\right).
\]
Thus,
\[
f\left(x\right)+g\left(Lx+y\right)\geq\alpha_\varepsilon +\psi_{x}\left(y\right)-\psi_{x}\left(0\right)\geq\alpha-\varepsilon/2,
\]
for all $y\in W\left(0\right)$. The inequality holds for all $x\in X$,
we can take the infimum with respect to $x\in X$ on both sides and
obtain 
\[
V\left(y\right)=\inf_{x\in X}f\left(x\right)+g\left(Lx+y\right)\geq\alpha_\varepsilon -\varepsilon/2=V\left(0\right)-\varepsilon,
\]
for all $y\in W\left(0\right)$. Hence, $V\left(y\right)$ is lower
semi-continuous at $y=0$.

(iii)$\Rightarrow$(ii):
Conversely, for every $\alpha \in \mathbb{R}, \alpha < \inf_{x\in X} \sup_{\psi \in \Psi} \mathcal{L} (x,\psi) = V(0)$, we choose  $3\varepsilon = V(0)-\alpha >0$. There exists a neighborhood
$W\left(0\right) \subset Y$ such that 
\[
V\left(0\right)-\varepsilon\leq V\left(y\right),\quad\forall y\in W\left(0\right).
\]
Now considering 
\begin{align*}
f\left(x\right)+g_{\Psi}^{**}\left(Lx+y\right) & =\sup_{\psi\in\Psi}f\left(x\right)+\psi\left(Lx+y\right)-g_{\Psi}^{*}\left(\psi\right)
%& \leq\sup_{\psi\in\Psi}f\left(x\right)-g_{\Psi}^{*}\left(\psi\right)+\psi\left(Lx\right)+\varepsilon.
\end{align*}
We can find $\psi_\varepsilon\in \Psi$ such that 
\begin{align*}
f\left(x\right)+g_{\Psi}^{**}\left(Lx+y\right) & < f\left(x\right)+\psi_\varepsilon\left(Lx+y\right)-g_{\Psi}^{*}\left(\psi_\varepsilon\right) +\varepsilon\\
 & \leq f\left(x\right)-g_{\Psi}^{*}\left(\psi_\varepsilon\right)+\psi_\varepsilon\left(Lx\right)+2\varepsilon, \quad \forall y\in W_1 (0).
\end{align*}
In the last inequality, we can find a neighborhood $W_1 (0)\subset Y$ such that $\psi_\varepsilon$ is continuous at $0$. Hence, for all $y\in W_1 (0) \cap W(0)$, we have 
\begin{align*}
V\left(0\right)-\varepsilon & \leq V\left(y\right) = \inf_{x\in X}f\left(x\right)+g_{\Psi}^{**}\left(Lx+y\right)\\
& \leq\sup_{\psi\in\Psi}f\left(x\right)-g_{\Psi}^{*}\left(\psi\right)+\psi\left(Lx+y\right), \quad \forall x\in X\\
& < f\left(x\right)-g_{\Psi}^{*}\left(\psi_\varepsilon\right)+\psi_\varepsilon\left(Lx+y\right)+\varepsilon\\
& \leq f\left(x\right)-g_{\Psi}^{*}\left(\psi_\varepsilon\right)+\psi_\varepsilon\left(Lx\right)+2\varepsilon, \quad \forall y\in W_1 (0)\\
& =\mathcal{L}\left(x,\psi_{\varepsilon}\right)+2\varepsilon.
\end{align*}
Finally, $\alpha =V(0) -3\varepsilon \leq \mathcal{L} (x,\psi_\varepsilon)$ for all $x\in X$, we can take $\phi=\alpha\in\Phi$ and $\phi\in \text{supp } \mathcal{L}\left(\cdot,\psi_{\varepsilon}\right)$. Then $\phi$ has the intersection property at level
$\alpha$ with any $\phi_{1}\in \text{supp } \mathcal{L}\left(\cdot,\psi\right)$, for $\psi\in \Psi$.
\end{proof}
Now we can state zero duality gap for Lagrange dual problem.

\begin{proposition}
\label{prop:intersection-1}
Let $f:X\to\left(-\infty,+\infty\right],g:Y\to\left(-\infty,+\infty\right]$. Let $L:X\to Y$ be a mapping from $X$ to $Y$ with $\text{dom } g\cap L\left(\text{dom }f\right)\neq\emptyset$. Let $\Phi,\Psi$ be sets of elementary functions defined on  $X$ and $Y$, respectively. Let $\mathcal{L} (x,\psi)$ be the Lagrangian defined by \eqref{eq:Lagrangian}. Assume $\Psi$ is
convex, $0\in\Phi$ and \eqref{eq:CP-LP cond 1} holds. We further assume $\inf_{x\in X} f(x) +g(Lx) <+\infty$. The following are equivalent.
\begin{enumerate}
\item For every $\alpha<\inf_{x\in X}\sup_{\psi\in\Psi}\mathcal{L}\left(x,\psi\right)$,
there exists $\psi_{1},\psi_{2}\in\Psi$ and $\bar{\phi}_{1}\in\text{supp }\mathcal{L}\left(\cdot,\psi_{1}\right),\bar{\phi}_{2}\in\text{supp }\mathcal{L}\left(\cdot,\psi_{2}\right)$
such that $\bar{\phi}_{1},\bar{\phi}_{2}$ have the intersection property at level $\alpha$.
\item $\displaystyle{\inf_{x\in X} f(x) + g(Lx) = \inf_{x\in X}\sup_{\psi\in\Psi} \mathcal{L}\left(x,\psi\right)=\sup_{\psi\in\Psi}\inf_{x\in X}\mathcal{L}\left(x,\psi\right)< +\infty}.$
\end{enumerate}
\end{proposition}
\begin{proof}
By applying Theorem \ref{thm:Lagrange intersection} and Theorem \ref{thm:CP zero gap}, we obtain the assertion.
\end{proof}

\begin{remark}
\
\begin{itemize}
\item In order to work with the intersection property, we need to find $\psi_1,\psi_2\in\Psi$ and $\phi_{1}\in\text{supp }\mathcal{L}\left( \cdot,\psi_1 \right),\phi_{2}\in\mathcal{L}\left( \cdot,\psi_2 \right)$. Because $\alpha<\inf_{x\in X}\sup_{\psi\in\Psi}\mathcal{L}\left(x,\psi\right)$, so $\alpha\in\text{supp }\mathcal{L}\left(\cdot,\psi\right)$
for any $\psi\in\Psi$. We can set $\phi_{1}=\alpha\in\Phi$, and the intersection property always holds for any $\phi_2\in\text{supp }\mathcal{L}\left( \cdot,\psi_2\right)$, as $\left[\phi_1<\alpha\right]=\emptyset$.

\item Let $\Psi$ be symmetric, and for every $\varepsilon>0$, there exist $\bar{x}\in X$, $\psi_1\circ L \in \Phi$ for $\psi_1\in\Psi$, such that $\psi_{1} \circ L\in\partial_{\varepsilon, \Phi} f\left( \bar{x}\right)$. We have, for all $y\in X$
\begin{align*}
f\left(y\right)-f\left(\bar{x}\right) & \geq\psi_{1}\left(L y\right)-\psi_{1}\left(L \bar{x}\right)-\varepsilon\\
f\left(y\right)-\psi_{1}\left(L y\right)-g^{*}_\Psi\left(-\psi_{1}\right) & \geq f\left(\bar{x}\right)-\psi_{1}\left(L \bar{x}\right)-g^{*}_\Psi\left(-\psi_{1}\right)-\varepsilon\\
\mathcal{L}\left(y,-\psi_{1}\right) & \geq\mathcal{L}\left(\bar{x},-\psi_{1}\right)-\varepsilon,
\end{align*}
so $\bar{x}$ is a $\varepsilon$-approximate solution to the problem
$\inf_{x\in X}\mathcal{L}\left(x,-\psi_{1}\right)$.
\end{itemize}
\end{remark}

We can replace assumption \eqref{eq:CP-LP cond 1} with a stronger one by the following lemma.
\begin{lemma}
Assume $\inf_{x\in X} f(x) + g(Lx) <+\infty$ and let $\mathcal{L} (x,\psi)$ be the Lagrangian function given by \eqref{eq:Lagrangian}. Then condition \eqref{eq:CP-LP cond 1} i.e.,
\[
\inf_{x\in X} f(x) + g(Lx) =\inf_{x\in X} f(x) + g^{**}_\Psi (Lx),
\]
holds if and only if for every $\varepsilon\geq 0$, there exists $x_\varepsilon \in X$ such that 
\begin{equation}
\label{lm:LP intersection prop 2}    
\inf_{x\in X} \sup_{\psi\in\Psi} \mathcal{L} (x,\psi) >f(x_\varepsilon) + g(Lx_\varepsilon) -\varepsilon.
\end{equation}
\end{lemma}
\begin{proof}
If condition \eqref{eq:CP-LP cond 1} holds i.e., 
\[
\inf_{x\in X} f(x) + g(Lx) =\inf_{x\in X} f(x) + g^{**}_\Psi (Lx),
\]
then for every $\varepsilon>0$, one can find an $x_\varepsilon\in X$ such that $\inf_{x\in X} f(x) + g(Lx) > f(x_\varepsilon) + g(Lx_\varepsilon) -\varepsilon$. Hence,
\[
\inf_{x\in X} \sup_{\psi\in\Psi} \mathcal{L} (x,\psi) = \inf_{x\in X} f(x) + g^{**}_\Psi (Lx) > f(x_\varepsilon) + g(Lx_\varepsilon) -\varepsilon.
\]
Conversely, assume that 
\[
\inf_{x\in X} \sup_{\psi\in\Psi} \mathcal{L} (x,\psi) > f(x_\varepsilon) + g(Lx_\varepsilon) -\varepsilon.
\]
We have 
\[
\inf_{x\in X} \sup_{\psi\in\Psi} \mathcal{L} (x,\psi) > f(x_\varepsilon) + g(Lx_\varepsilon) -\varepsilon \geq \inf_{x\in X} f(x) +g(Lx) -\varepsilon.
\]
Recall that,
\[\inf_{x\in X} \sup_{\psi\in\Psi} \mathcal{L} (x,\psi) = \inf_{x\in X} f(x) +g^{**}_\Psi (Lx).
\]
Because $g(Lx) \geq g^{**}_\Psi (Lx)$ for all $x\in X$, we can write
\[
\inf_{x\in X} f(x) + g(Lx) \geq \inf_{x\in X} f(x) +g^{**}_\Psi (Lx) \geq \inf_{x\in X} f(x) +g(Lx) -\varepsilon.
\]
Both sides do not depend on $\varepsilon$, we can let $\varepsilon\to 0$ and get the equality.
\end{proof}

In fact, condition \eqref{lm:LP intersection prop 2} is stronger than intersection property, as we can see in the corollary below.
\begin{corollary}
\label{cor:Lagrange zero gap best}
Let $f:X\to\left(-\infty,+\infty\right],g:Y\to\left(-\infty,+\infty\right]$
and let $L:X\to Y$ be a mapping from $X$ to $Y$ with $\text{dom } g\cap L\left(\text{dom }f\right)\neq\emptyset$. Let $\Phi,\Psi$ be sets of elementary functions defined on $X$ and $Y$, respectively. Let the Lagrangian function be defined by \eqref{eq:Lagrangian}. Assume $\Psi$ is
convex, $0\in\Phi$, $\inf_{x\in X} f(x) +g(Lx) <+\infty$. The following are equivalent. 
\begin{enumerate}
\item For every $\varepsilon>0$, there exists $x_\varepsilon\in X$ such that $\displaystyle{\inf_{x\in X} \sup_{\psi\in \Psi} \mathcal{L} (x,\psi)} \geq f(x_\varepsilon) + g(Lx_\varepsilon) -\varepsilon$.
\item For every $\varepsilon>0$, there exist $x_\varepsilon\in X$ and $\psi_\varepsilon\in\partial_{\varepsilon, \Psi}g\left(Lx_\varepsilon\right)$
such that $0\in\partial_{\varepsilon, \Phi}\left(f+\psi_\varepsilon\circ L\right)\left(x_\varepsilon\right)$.
\item $\displaystyle{\inf_{x\in X} f(x) + g(Lx) = \inf_{x\in X}\sup_{\psi\in\Psi} \mathcal{L}\left(x,\psi\right)=\sup_{\psi\in\Psi}\inf_{x\in X}\mathcal{L}\left(x,\psi\right)< +\infty}.$
\end{enumerate}
\end{corollary}

\begin{proof}
Thanks to Theorem \ref{thm:CP zero gap}, Theorem \ref{thm:Lagrange intersection} and Proposition \ref{prop:intersection-1}, we have (ii)$\Leftrightarrow$ (iii). We only need to prove (i)$\Leftrightarrow$(ii).

(i) $\Rightarrow$ (ii): For every $\varepsilon>0$, we have 
\[
\sup_{\psi\in\Psi} \mathcal{L}(x,\psi) \geq \inf_{x\in X} \sup_{\psi\in \Psi} \mathcal{L} (x,\psi) \geq f(x_\varepsilon) + g(Lx_\varepsilon) -\varepsilon.
\]
There also exists $\psi_\varepsilon \in \Psi$ such that 
\[
\mathcal{L}(x,\psi_\varepsilon) +\varepsilon \geq \sup_{\psi\in\Psi} \mathcal{L}(x,\psi) \geq f(x_\varepsilon) + g(Lx_\varepsilon) -\varepsilon.
\]
Thus, $\mathcal{L}(x,\psi_\varepsilon) +\varepsilon \geq  f(x_\varepsilon) + g(Lx_\varepsilon) -\varepsilon$ for all $x\in X$. While $\mathcal{L} (x,\psi_\varepsilon) = f(x) + \psi_\varepsilon (Lx) - g^*_\Psi (\psi_\varepsilon)$, we have
\begin{equation}
\label{eq:cor intersection prop 1}
(\forall x\in X) \ f(x) + \psi_\varepsilon (Lx) - g^*_\Psi (\psi_\varepsilon) +\varepsilon \geq f(x_\varepsilon) + g(Lx_\varepsilon) -\varepsilon.
\end{equation}

Let $x = x_\varepsilon$ and we obtain 
\[
\psi_\varepsilon (Lx_\varepsilon) - g^*_\Psi (\psi_\varepsilon) \geq g(Lx_\varepsilon) -2\varepsilon,
\]
so $\psi_\varepsilon\in \partial_{2\varepsilon, \Psi} g(Lx_\varepsilon)$. On the other hand, by using the property of $\Phi$-conjugate,
\[
(\forall x\in X) \ g(Lx) + g^*_\Psi (\psi_\varepsilon) \geq \psi_\varepsilon (Lx).
\]
By putting this in \eqref{eq:cor intersection prop 1}, we have $0\in \partial_{2\varepsilon, \Phi} (f+\psi_\varepsilon\circ L) (x_\varepsilon)$. Hence, we prove (ii).

(ii) $\Rightarrow$ (i): For every $\varepsilon>0$, there exist $x_\varepsilon\in X$, $\psi_\varepsilon \in \partial_{\varepsilon, \Psi} g(Lx_\varepsilon)$ and $0\in \partial_{\varepsilon, \Phi} (f+\psi_\varepsilon\circ L) (x_\varepsilon)$, i.e. for all $z\in X$
\[
f(z) + \psi_\varepsilon(Lz) - f(x_\varepsilon) - \psi_\varepsilon(Lx_\varepsilon) \geq -\varepsilon.
\]
Because $\psi_\varepsilon\in\partial_{\varepsilon, \Psi} g(Lx_\varepsilon)$, using inequality \eqref{prop1:Fenchel} we get, 
\begin{align*}
    f(z) + \psi_\varepsilon(Lz) - f(x_\varepsilon) &\geq \psi_\varepsilon(Lx_\varepsilon) -\varepsilon \\
    & \geq g^*_\Psi (\psi_\varepsilon) + g(Lx_\varepsilon) - 2\varepsilon,
\end{align*}
and
\[
\sup_{\psi\in\Psi} \mathcal{L} (z,\psi) \geq \mathcal{L} (z,\psi_\varepsilon) = f(z) + \psi_\varepsilon(Lz) - g^*_\Psi (\psi_\varepsilon) \geq f(x_\varepsilon) + g(Lx_\varepsilon) - 2\varepsilon.
\]

The right-hand side does not depend on $z\in X$, taking the infimum with respect to $z\in X$ gives us
\[
\inf_{z\in X} \sup_{\psi\in\Psi} \mathcal{L} (z,\psi) \geq f(x_\varepsilon) + g(Lx_\varepsilon) - 2\varepsilon.
\]
We have proved (i).
\end{proof}

Having the class $\Psi$ of elementary functions, we recall the notion of $X$-convexity in \cite[Proposition 1.2.3]{Pall2013}, by using fomula $\phi(x) = x(\phi)$ for $x\in X,\phi\in\Phi$. Together with the intersection property, we can prove Langrage zero duality gap as follow.
\begin{proposition}
\label{prop:new intersection}
Assume that for every $\alpha\in\mathbb{R}, \alpha\leq\inf_{x\in X}f\left(x\right)+\psi\left(Lx\right)$, there exist $\psi\in\text{supp }g,x_0\in X$, such that $f(x_0)\leq \alpha \leq \psi (Lx_0)$. If $x_0 \in A=\left\{x\in X: Lx\in\text{supp }g^{*}_\Psi\right\} \neq \emptyset$, then we have $val (LP) = val (LD)$.
\end{proposition}
\begin{proof}
Let $x_{0}\in A, \psi\in\text{supp }g$ be such that 
\[
f(x_0) \leq \alpha \leq \psi (Lx_0).
\]
Then $g^{*}_\Psi\left(\psi\right)\leq0$. On the other hand, $x_0\in A$ means 
\[
(\forall \psi\in\Psi) \ \psi (Lx_0) \leq g^*_\Psi (\psi) \Leftrightarrow \sup_{\psi\in\Psi}\psi\left(Lx_{0}\right)-g^{*}_\Psi\left(\psi\right)\leq0.
\]
Moreover, $\psi\left(Lx_{0}\right)\geq\alpha$ and
\begin{align*}
\alpha & \leq\inf_{x\in X}f\left(x\right)+\psi\left(Lx\right)\\
 & \leq\inf_{x\in X}f\left(x\right)+\psi\left(Lx\right)-g^{*}_\Psi\left(\psi\right),\quad\text{as }g^{*}_\Psi\left(\psi\right)\leq0\\
 & \leq\sup_{\psi\in\Psi}\inf_{x\in X}f\left(x\right)+\psi\left(Lx\right)-g^{*}_\Psi\left(\psi\right).
\end{align*}
And,
\begin{align*}
\alpha & \geq f\left(x_{0}\right)\geq\sup_{\psi\in\Psi}f\left(x_{0}\right)+\psi\left(Lx_{0}\right)-g^{*}_\Psi\left(\psi\right)\\
 & \geq\inf_{x\in X}\sup_{\psi\in\Psi}f\left(x\right)+\psi\left(Lx\right)-g^{*}_\Psi\left(\psi\right).
\end{align*}
Combining the two above inequalities, we obtain
\[
\sup_{\psi\in\Psi}\inf_{x\in X}f\left(x\right)+\psi\left(Lx\right)-g^{*}_\Psi\left(\psi\right) \geq \inf_{x\in X}\sup_{\psi\in\Psi}f\left(x\right)+\psi\left(Lx\right)-g^{*}_\Psi\left(\psi\right).
\]
\end{proof}

\begin{remark} 
%\begin{itemize}\ 
In Proposition \ref{prop:new intersection}, we can replace the assumption $x_0\in A$ with $g(Lx_0) \leq 0$, for $x_0\in X$. Then zero duality gap holds between \eqref{prob:CP} and \eqref{prob:LD}. Because, for $\psi_0 \in \text{supp } g$ and $\psi_0 \geq \alpha \geq f(x_0)$, we have
\begin{align*}
\sup_{\psi\in\Psi}\inf_{x\in X}f\left(x\right)+\psi\left(Lx\right)-g^{*}_\Psi\left(\psi\right) &\geq \inf_{x\in X}f\left(x\right)+\psi_0\left(Lx\right)-g^{*}_\Psi\left(\psi_0\right)\\
&\geq 
\inf_{x\in X}f\left(x\right)+\psi_0\left(Lx\right) ,\quad\text{as }g^{*}_\Psi\left(\psi_0\right)\leq0\\
& \geq \alpha \geq f(x_0) \geq f(x_0) + g(Lx_0) \\
& \geq \inf_{x\in X} f(x) +g(Lx).
\end{align*}
%\item For simplicity, consider $L=Id$ and $X=Y$. If $\psi_{0}\in\text{supp }g$ then $g^*_\Psi (\psi_0) \leq 0$. 
%we have 
%\begin{align*}
%\psi_{0}\left(y\right) & \leq g\left(y\right)\quad\left(\forall y\in X\right)\\
%\psi_{0}\left(y\right)-g\left(y\right) & \leq0\\
%g^{*}_\Psi\left(\psi_{0}\right) & \leq0.
%\end{align*}
%{\color{blue} On the other hand if $y_{0}\in\text{supp }g^{*}_\Psi$ then $\psi\left(y_{0}\right) \leq g^{*}_\Psi\left(\psi\right)$ for all $\psi \in \Psi$. Taking $\psi=\psi_0$ and we get $\psi_{0}\left(y_{0}\right) \leq0$
%\begin{align*}
%\psi\left(y_{0}\right) & \leq g^{*}_\Psi\left(\psi\right)\quad\left(\forall\psi\in\Psi\right)\\
%\psi_{0}\left(y_{0}\right) & \leq0\text{ for }\psi=\psi_{0}.
%\end{align*}

%We also need 
%\begin{align*}
%\alpha & \leq\inf_{x\in X}f\left(x\right)+\psi_{0}\left(x\right)\\
%f\left(y_{0}\right) & \leq \alpha\leq\psi\left(y_{0}\right).
%\end{align*}

%This means 
%\[
%f\left(y_{0}\right)\leq \alpha\leq f\left(y_{0}\right)+\psi_{0}\left(y_{0}\right).
%\]
%If $f\left(y_{0}\right)\leq\alpha$
%then $\psi_{0}\left(y_{0}\right)=0$ and the result is still valid. Then
%the choice of $\alpha$ has to be nonpositive and this can fail if
%$f>0$. }
%\end{itemize}
\end{remark}
We complete this Section with two examples which illustrate Proposition \ref{prop:new intersection}.
\begin{example}
Consider the functions $f\left(x\right)=3x^{2}-3x-10,g\left(x\right)=-2x^{2}+x-8$ and $L=Id$. Let
\[
\Phi=\Psi=\left\{ \psi\left(x\right)=-ax^{2}+bx+c,a\geq0,b,c\in\mathbb{R}\right\} ,
\]
is our sets of elementary functions. We want to find $\text{supp }g$, i.e. the set of all $\psi\in \Psi$ such that 
\[
\psi\left(x\right)\leq g\left(x\right)\quad\forall x\in\mathbb{R}.
\]
This gives us $\psi\in \Psi$ where $a>2,c\leq-\frac{\left(1-b\right)^{2}}{4\left(a-2\right)}-8$
or $\psi\left(x\right)=-2x^{2}+x-8$.
Consider $a>2$, for every $\alpha \in \mathbb{R}$ such that
\[
\alpha\leq \inf_{x}f\left(x\right)+\psi\left(x\right)=\begin{cases}
c-10 & \text{if }a=3,b=3,c\leq-9\\
-\frac{\left(b-3\right)^{2}}{4\left(3-a\right)}-10+c & \text{if }2<a<3,c\leq-\frac{\left(1-b\right)^{2}}{4\left(a-2\right)}-8 \\
-\infty & \text{if } a>3.
\end{cases}
\]
We need to find $x_{0}$ such that $f\left(x_{0}\right)\leq \alpha\leq\psi\left(x_{0}\right)$,
we can let $\alpha=\inf_{x}f\left(x\right)+\psi\left(x\right)$. There are two cases
\begin{itemize}
\item $\alpha=c-10$ or $\alpha\leq-19$, we cannot find $x_{0}$
such that $f\left(x_{0}\right)\leq -19$ as $f_{\min}=-\frac{43}{4}$.

\item $\alpha=-\frac{\left(b-3\right)^{2}}{4\left(3-a\right)}-10+c,$
we have to solve $f(x_0) \leq \alpha \leq \psi (x_0)$ for $x_0$. We arrive at the following system 
\begin{align*}
-ax^{2}+bx+c & \geq-\frac{\left(b-3\right)^{2}}{4\left(3-a\right)}-10+c\\
3x^{2}-3x-10 & \leq -\frac{\left(b-3\right)^{2}}{4\left(3-a\right)}-10+c\\
2<a<3, & \ c\leq-\frac{\left(1-b\right)^{2}}{4\left(a-2\right)}-8.
\end{align*}
By direct calculations, there are no solution $a,b,c,x$ for the above system of inequalities.

For $\psi\left(x\right)=-2x^{2}+x-8$, and this turns into
\[
\inf_{x\in\mathbb{R}}f\left(x\right)+g\left(x\right) =\alpha =\inf_{x\in\mathbb{R}}f\left(x\right)+\psi\left(x\right)\leq \sup_{\psi\in\psi}\inf_{x\in\mathbb{R}}f\left(x\right)+\psi\left(x\right),
\]
which is zero duality gap.
\end{itemize}
\end{example} 

\begin{example}
Now consider the functions $f\left(x\right)=x^{2}-3x-10,g\left(x\right)=2x+1$ with $L=Id$,
and the set of elementary functions is
\[
\Phi=\Psi=\left\{ \psi\left(x\right)=ax+b,\text{ where }a,b\in\mathbb{R}\right\} .
\]
Calculating $\psi\in\text{supp }g$ gives us $a=2$ and $b\leq1$. Now we find
\[
\alpha = \inf_{x\in\mathbb{R}}f\left(x\right)+\psi\left(x\right)=\inf_{x\in\mathbb{R}}x^{2}-x-10+b=-\frac{21}{2}+b.
\]
We also calculate 
\[
g^{*}_\Psi\left(\psi\right)=\sup_{x\in X}\psi\left(x\right)-g\left(x\right)=b-1,
\]
for $b\leq1$. Observe that 
\begin{align*}
\sup_{\psi\in\text{ supp g}}\inf_{x\in X}f\left(x\right)+\psi\left(x\right)-g^{*}_\Psi\left(\psi\right) & =\sup_{\psi\in\text{ supp g} }\inf_{x\in X}x^{2}-x-10+b-b+1\\
 & =\sup_{\psi\in\text{ supp g} }\inf_{x\in X}x^{2}-x-9\\
 & =\inf_{x\in X}x^{2}-x-9=\inf_{x\in\mathbb{R}}f\left(x\right)+g\left(x\right),
\end{align*}
while 
\[
\sup_{\psi\in\text{ supp g}}\inf_{x\in X}f\left(x\right)+\psi\left(x\right)-g^{*}_\Psi\left(\psi\right) \leq \sup_{\psi\in\Psi }\inf_{x\in X}f\left(x\right)+\psi\left(x\right)-g^{*}_\Psi\left(\psi\right).
\]
Thus we have zero duality gap.
\end{example}

\section*{Acknowledgments}
The authors are thankful to the anonymous referees for their constructive comments and remarks which improve the quality of the paper.

This work has been supported by the ITN-ETN project TraDE-OPT funded by the European Union{'}s Horizon 2020 research and innovation programme under the Marie Sk{\l}odowska-Curie grant agreement No.861137
\bibliographystyle{tfnlm}
\bibliography{ref_library.bib}
\end{document}